\documentclass[11pt]{amsart}

\usepackage[a4paper,margin=1.15in]{geometry}
\usepackage{amsmath,amssymb,amsthm,mathtools}
\usepackage{esint}
\usepackage{mathrsfs}
\usepackage{enumitem}
\usepackage{microtype}
\usepackage[hidelinks]{hyperref}
\usepackage{color}

\allowdisplaybreaks
\numberwithin{equation}{section}
\newtheorem{theorem}{Theorem}[section]
\newtheorem{proposition}[theorem]{Proposition}
\newtheorem{lemma}[theorem]{Lemma}
\newtheorem{corollary}[theorem]{Corollary}
\theoremstyle{remark}
\newtheorem*{remark}{Remark}

\newcommand{\R}{\mathbb R}
\newcommand{\Sn}{\mathbb S^{n-1}}

\newcommand{\osc}{\operatorname{osc}}
\newcommand{\dist}{\operatorname{dist}}
\newcommand{\PV}{\operatorname{P.V.}}

\newcommand{\E}{\mathcal E}

\title[A quantitative Gidas--Ni--Nirenberg theorem]{A Quantitative Gidas--Ni--Nirenberg Theorem for the Fractional Laplacian}
\author[Q.~Shen]{Qizhen Shen}
\address{School of Mathematical Sciences\\
Shanghai Jiao Tong University\\
Shanghai, 200240, China} 
\email{zhuimengtiantang@sjtu.edu.cn}
\author[J.~Xie]{Jiongduo Xie}
\address{School of Mathematical Sciences\\
Shanghai Jiao Tong University\\
Shanghai, 200240, China}
\email{jiongduoxie@outlook.com}
\date{}

\subjclass[2020]{35R11, 35B06, 35B35, 35B50}
\keywords{Fractional Laplacian, quantitative symmetry, Gidas--Ni--Nirenberg theorem, moving planes, antisymmetric Harnack inequality, stability}

\begin{document}

\begin{abstract}
We establish a quantitative Gidas--Ni--Nirenberg theorem for the pure fractional
Dirichlet problem
\[
(-\Delta)^s u=k(x)g(u)\quad\text{in }B_1,
\qquad
u=0\quad\text{in }\mathbb R^n\setminus B_1,
\qquad 0<s<1.
\]
For positive bounded solutions under a two-sided $L^\infty$ normalization, we prove
\[
\mathcal D(u)\le C\,\mathcal D(k)^\gamma,
\qquad 0<\gamma<1,
\]
where $\mathcal D$ measures spherical oscillation and outward radial increase. %Thus, a purely zero-order defect of the coefficient quantitatively controls both radial symmetry and radial monotonicity, without requiring derivatives of $k$. 
To the best of our knowledge, this provides the first quantitative Gidas--Ni--Nirenberg estimate for the pure fractional Dirichlet problem. The key new ingredient is a shift-uniform, weak Harnack inequality for shifted antisymmetric supersolutions, which overcomes the exterior-tail obstruction in the quantitative moving-plane argument. This mechanism extends to nonseparable nonlinearities.
\end{abstract}
\maketitle
\enlargethispage{3pt}

\section{Introduction}

In recent years, there has been growing interest in qualitative studies such as symmetry and monotonicity \cite{BN91,CGS89,CFMN18,CGY26,CL91,DKW11,GNN79,GNN81,Li96,Savin09} for local and nonlocal elliptic equations. The moving-plane method is a fundamental technique based on reflection arguments for deriving rigidity from geometric invariance in elliptic problems. This reflection argument was first introduced by Alexandrov in his study of constant-mean-curvature hypersurfaces \cite{Alex62}. Serrin then applied it to overdetermined problems \cite{Serrin71}. Later, Gidas--Ni--Nirenberg used it to prove radial symmetry for positive solutions in balls \cite{GNN79,GNN81}. Important refinements include \cite{BN91,CL91,Dancer92}. For a quantitative perspective, see \cite{CR18}. When the symmetry of the problem is destroyed, due to a small perturbation of the operator, the domain, or the nonlinear term, the quantitative deviation of the solution from symmetry has recently attracted considerable interest in both the local \cite{CCPP24,CL25,Gatti25} and the mixed local–nonlocal settings \cite{Biswas26}. However, such results for pure nonlocal problems remain unexplored to the best of our knowledge.

Our main aim in this paper is to establish the quantitative symmetry of positive solutions for the Dirichlet problem involving the fractional Laplacian. The study of quantitative symmetry in nonlocal problems has arisen mainly from geometric rigidity, including sharp stability for nonlocal mean curvature \cite{CFMN18} and the overdetermined problem \cite{CDPPV23,DPTV23Serrin,GSW26}.  For the qualitative background of nonlocal equations, we refer to \cite{BMS18,CLL17,CL005,CLO06,CLZ17,DQ18,Dou16,FQJ12,FW14,CDPPV23,Jarohs16,JW16,LZ11}.
 
%For fractional equations, the function on one side of a hyperplane is connected to the other side through the full kernel. Thus, the reflection method is inherently global. 
%Early work on fractional symmetry includes \cite{BLW05,CL005,CLO06}. Later, moving-plane and antisymmetric maximum-principle methods were developed for fractional Laplacian and more general nonlocal operators \cite{BMS18,CL18,CLL17,CLZ17,DQ18,Dou16,FQJ12,FW14,GSGL23,Jarohs16,JW16,LZ11}. Recent work has clarified the role of antisymmetric functions and their Harnack theory \cite{DKTV25,DPTV24,DTV23}. Related reflection principles also apply to other operators, such as the logarithmic Laplacian \cite{PS25}. A recent survey of qualitative methods for fractional equations is \cite{CGL26}. These results provide the qualitative background for this paper.

\subsection{The model problem and the solution class}
\par
The fractional Laplacian has recently attracted considerable attention as a modeling tool for a wide range of physical phenomena involving long-range interactions. These include anomalous diffusion, quasi-geostrophic flows, turbulence, water waves, and molecular dynamics (see \cite{BG90, CV10, C06, TZ06} and the references therein). In particular, it appears as the infinitesimal generator of a stable Lévy diffusion process \cite{Be96}, and thus plays a key role in describing anomalous diffusion observed in plasmas and other settings.

The fractional Laplacian can be defined as
\begin{eqnarray}\label{eq1-1}
\begin{aligned}
(-\Delta)^s u(x)&=c_{n, s}\,\PV\! \int_{\mathbb{R}^n}\frac{u(x)-u(y)}{|x-y|^{n+2s}}\,dy,\\
%&=C_{n, s}PV \lim_{\varepsilon \to 0}\int_{\mathbb{R}^n\backslash B_\varepsilon(x)}\frac{u(x)-u(y)}{|x-y|^{n+2s}}dy,
\end{aligned}
\end{eqnarray}
where $0<s<1$, $\,\PV\!$ stands for the Cauchy principal value and
%$B_{\varepsilon}(x)$ is the ball of radius $\varepsilon$ centered at $x,$ and 
$c_{n, s}$ is a dimensional constant that depends on $n$ and $s$, precisely given by
$$
c_{n,s} = \frac{2^{2s} s \Gamma\left(\frac{n}{2}+s\right)}{\pi^{\frac{n}{2}}\Gamma(1-s)}.
$$
Here, $\Gamma(\cdot)$ denotes the Gamma function. For \eqref{eq1-1} to be pointwise well-defined, we assume $u\in C^{1,1}_{loc}(\mathbb{R}^n)\cap \mathcal{L}_{2s}$, where
$$
\mathcal{L}_{2s}=\left\{u\in L_{loc}^1 \,\Big| \int_{\mathbb{R}^n}\frac{|u(x)|}{1+|x|^{n+2s}}dx<\infty \right\}.
$$

In this paper, we study the Dirichlet problem
\begin{equation}\label{eq:main}
\begin{cases}
(-\Delta)^s u=k(x)g(u),&x\in B_1,\\
u>0,&x\in B_1,\\
u=0,&x\in\R^n\setminus B_1,
\end{cases}
\end{equation}
where $n\ge2$ and $0<s<1$. Set
\[
X_0^s(B_1):=\{u\in H^s(\R^n):u=0\ \text{a.e.\ in }\R^n\setminus B_1\}.
\]
We call $u$ a classical solution of \eqref{eq:main} if $u\in X_0^s(B_1)\cap C^{1,1}_{\rm loc}(B_1)$ satisfies \eqref{eq:main}.

In the following, we denote
$t_+:=\max\{t,0\}$ and $t_-:=\max\{-t,0\}.$
For a measurable set $E$, we define
$$\osc_Ev:=\sup_Ev-\inf_Ev.$$ 

For a bounded function $v$ in $B_1$, the \emph{defect} of $v$ is defined by 
\begin{equation}\label{eq:raddef}
\mathcal D(v)
:=\sup_{0<r<1}\osc_{\partial B_r}v
+\sup_{e\in\Sn}\sup_{0\le\rho<r<1}
\bigl(v(re)-v(\rho e)\bigr)_+.
\end{equation}
Its first term measures angular
oscillation and its second term measures increase along rays from the origin. Thus,
\[
\mathcal D(v)=0
\quad\Longleftrightarrow\quad
v\ \text{is radial and nonincreasing in }|x|.
\]

Our main result is the following.

\begin{theorem}[Quantitative radial symmetry and monotonicity]\label{thm:main}
Let $n\ge2$, $0<s<1$, 
$
k\in C(B_1;[0,\infty))\cap L^\infty(B_1)
$ and $
g\in C^{0,1}_{\mathrm{loc}}([0,\infty);[0,\infty)).
$
Assume $u$ to be a classical solution of \eqref{eq:main} satisfying
\begin{equation}\label{eq:norm}
\frac1{C_0}\le \|u\|_{L^\infty(B_1)}\le C_0,
\end{equation}
for some $C_0\ge1$. Then there exist $C>0$ and $\gamma\in(0,1)$, depending only on $n,\ s,\ C_0,\ $ $\|g\|_{C^{0,1}([0,C_0])}$ and $\|k\|_{L^\infty(B_1)},$ such that
\begin{equation}\label{eq:mainest}
\mathcal D(u)\le C\,\mathcal D(k)^\gamma.
\end{equation}
\end{theorem}

The estimate \eqref{eq:mainest} controls both closeness to radial symmetry and a finite-difference one-sided form of radial monotonicity. Zero order is a strengthening compared to the local theorem of Ciraolo, Cozzi, Perugini, and Pollastro \cite{CCPP24}. Their zero-order coefficient quantity of the same type as \eqref{eq:raddef} is sufficient for almost radial symmetry, while their pointwise radial-derivative estimate uses first-order information. 

For $x\ne0$, set
\[
\partial_r k(x)=\frac{x}{|x|}\cdot\nabla k(x),\qquad\nabla_T k(x)=\nabla k(x)-\frac{x}{|x|}\partial_r k(x),
\]
and define
\[
\mathcal D_1(k):=\|\nabla_T k\|_{L^\infty(B_1)}+\|(\partial_r k)_+\|_{L^\infty(B_1)}.
\]
The quantity $\mathcal D_1(k)$ measures the first-order radial symmetry and monotonicity of $k$ in $B_1$. In the special case where $k$ is $C^1$, we then have the following corollary.

\begin{corollary}\label{cor:firstorder}
Under the assumptions of Theorem~\ref{thm:main}, suppose, in addition, that $k\in C^1(B_1)$ and $\nabla k\in L^\infty(B_1)$. Then
\[
\mathcal D(u)\le C\,\mathcal D_1(k)^\gamma,
\]
with $C$ and $\gamma$ depending only on the data in Theorem~\ref{thm:main}.
\end{corollary}

Another direct consequence of Theorem \ref{thm:main} is that, when $k$ is radially symmetric and decreasing, we obtain the radial symmetry and monotonicity of $u$. See \cite[Theorem 1.1]{Zhang15} and \cite{FW14,WQY24}.

When $g$ changes sign, the one-sided radial defect of $k$ no longer controls the defect of $u$. The following version uses the two-sided quantity $\osc_{B_1}k$ and a soft decay assumption of $u$ near the boundary.

\begin{theorem}\label{thm:signchanging}
Let $n\ge2$, $0<s<1$, $k\in C(B_1;[0,\infty))\cap L^\infty(B_1), g\in C^{0,1}_{\mathrm{loc}}([0,\infty);\R)$, and $F\in C([0,1];[0,\infty))$ satisfy
$F(0)=0$ and $ F(t)>0$ for all $t>0$.
Assume that $u$ is a classical solution of \eqref{eq:main} and for all $x\in B_1$,
\begin{equation}\label{eq:Fbound}
F(1-|x|)\le u(x)\le C_0.
\end{equation}
Then there exists a modulus of continuity $\omega_F$, depending only on $n,\ s,\ C_0,\ F,\ \|g\|_{C^{0,1}([0,C_0])},$\quad$\ \|k\|_{L^\infty(B_1)}$,
such that
\begin{equation}\label{eq:modulus}
\mathcal D(u)\le\omega_F(\osc_{B_1}k),\qquad \omega_F(t)\to0\quad\text{as }t\downarrow0.
\end{equation}
\end{theorem}
\begin{remark}
    In particular, under the stronger assumption of Theorem \ref{thm:main}, Lemma \ref{lem:lower} allows us to take $F(t)=Ct^{s}$. Similar argument of Theorem \ref{thm:signchanging} yields $w_{F}(t)=Ct^{\gamma},$ where $\gamma\in(0,1)$.
\end{remark}
Also, there are similar results hold for Theorem \ref{thm:signchanging}.
\begin{corollary}
Under the assumptions of Theorem~\ref{thm:signchanging}, suppose, in addition, that $k\in C^1(B_1)$ and $\nabla k\in L^\infty(B_1)$. Then
\[
\mathcal D(u)\le\omega_F(\|\nabla k\|_{L^\infty(B_1)}),\qquad \omega_F(t)\to0\quad\text{as }t\downarrow0,
\]
with $C$ and $w_{F}$ depending only on the data in Theorem~\ref{thm:signchanging}.
\end{corollary}
\begin{corollary}
 Under the assumptions of Theorem~\ref{thm:signchanging}, suppose, in addition, $k$ is constant in $B_{1}$. Then $u$ is symmetric and nonincreasing with respect to $|x|$ at origin.
\end{corollary}

%If $\mathcal D(k)=0$, then $k(x)=k_0(|x|)$ with $k_0$ nonincreasing. Extend $g$ to $\R$ by $\widetilde g(t)=g(t_+)$. The map $F(x,t)=k_0(|x|)\widetilde g(t)$ is locally Lipschitz in $t$ on bounded intervals and has the reflection monotonicity required in the qualitative nonlocal theory. Because our solution is nonnegative and belongs to $X_0^s(B_1)$, the qualitative theorem of Jarohs and Weth \cite[Theorem~1.1 and Corollary~1.3(ii)]{JW16} implies that $u$ is radial and nonincreasing in $|x|$. Related fractional symmetry results include \cite{BLW05,FW14,Dou16,CLL17,BMS18} and the fractional overdetermined reflection problem in \cite{FJ15}. Therefore, exact symmetry at zero defect is part of the qualitative theory. In contrast, the new content of Theorem~\ref{thm:main} is the uniform quantitative estimate when the defect is nonzero.

\subsection{Quantitative moving planes}
A quantitative moving-plane theory was initiated in overdetermined settings by Aftalion,  Busca and Reichel \cite{ABR99}. Later, It was developed in several directions. For instance, there is the Hölder stability result for Serrin's problem \cite{CMV16} and the sharp quantitative Alexandrov theorem of Ciraolo and Vezzoni \cite{CV18}. Rosset considered an approximate Gidas–Ni–Nirenberg theorem for domains close to a ball \cite{Rosset94}, while \cite{CCPP24} treats a fixed ball with a spatially perturbed semilinear coefficient. Further developments include  semilinear problems \cite{CCG25,CCPP24}, $p$-Laplace equations \cite{CL25,DGSPV25,Gatti25} and mixed local–nonlocal operator \cite{Biswas26}. We point out that the local part plays the main role in the argument of \cite{Biswas26}, while the nonlocal term does not cause any difficulty.

The obstruction for \eqref{eq:main} appears at the nearly critical plane. Indeed, if we generalize the weak Harnack inequalities \cite[Lemma~5.1]{Biswas26} and \cite[Lemma~2.2]{CCPP24} to the pure nonlocal setting, tail terms cannot be avoided when applying them to a function that may not be nonnegative in $\mathbb R^n$. Note that a quantitative comparison gives the reflected difference only up to an additive error of size $a$, i.e.,
\begin{equation}\label{eq:introshift}
v(x^\lambda)=2a-v(x),\qquad a\ge0,
\end{equation}
rather than genuine antisymmetry. Hence, we adjust the classical weak Harnack inequalities for antisymmetric functions \cite{DCKP14,DKTV25,DPTV24,DTV23} to more general function \eqref{eq:introshift},
%To overcome this difficulty, we introduce a tail-free weak Harnack inequality for shifted antisymmetric supersolutions
which allows us to remove the exterior-tail obstruction in the quantitative moving-plane argument. 

In addition, the narrow-region and small-volume maximum principles are important tools in the moving-plane method \cite{CLL17,JW16}. We establish quantitative versions of these principles for antisymmetric functions. Moreover, using the Green function of the ball \cite{Bucur16}, we establish a bounded decay scale $(1-|x|)^s$ for the solution with positive nonlinear term. 
%The critical-plane argument combines a boundary factor $\lambda_*^s$ with a propagation factor $\lambda_*^\beta$.

\subsection{Generalization}
The form of product $k(x)g(u)$ is chosen because it makes the coefficient defect and the comparison with \cite{CCPP24} transparent, but this choice is not a structural restriction of the proof. In fact, Section~\ref{sec:extensions} proves the corresponding theorem for the nonseparable equation $(-\Delta)^s u=f(x,u)$ and a signed version under an explicit boundary lower bound. 

Corresponding to \eqref{eq:raddef} and $\osc_{B_1}k$, for a bounded function $f:B_1\times[0,C_0]\to\mathbb R$, we define the spatial defect uniformly in the solution variable by
\begin{equation}\label{eq:fdef}
\mathcal D_x(f;C_0):=\sup_{0\le t\le C_0}\mathcal D(f(\cdot,t)),
\end{equation}
and the spatial oscillation uniformly in the solution variable by
\begin{equation}\label{eq:fosc}
\operatorname{osc}_x(f;C_0) :=\sup_{0\le t\le C_0}\osc_{B_1}f(\cdot,t).
\end{equation}
Then results analogous to Theorem~\ref{thm:main} and \ref{thm:signchanging} follow by using $\mathcal D_x(f;C_0)$ and $\operatorname{osc}_x(f;C_0)$. For more details, we refer to Theorem~\ref{thm:nonseparable} and Corollary~\ref{cor:signed-nonseparable}.

The paper is organized as follows. Section~2 proves the tail-free weak Harnack inequality for the shifted antisymmetry function, shows the quantitative antisymmetric maximum principles, and provide the boundary estimate and Green representation needed later. Section~3 proves Theorems~\ref{thm:main} and~\ref{thm:signchanging}. Section~\ref{sec:extensions} first proves two direct consequences, then treats nonseparable nonlinearities and discusses the power scale.

\section{Preliminaries}

\subsection{Tail-free weak Harnack inequality}

For $\lambda\in\R$, write
\[
H_\lambda^+:=\{x_1>\lambda\},\qquad H_\lambda^-:=\{x_1<\lambda\},\qquad T_\lambda:=\{x_1=\lambda\},\qquad x^\lambda:=(2\lambda-x_1,x_2,\dots,x_n).
\]
For $\lambda\in[0,1)$ set
\begin{equation}\label{eq:Sigma}
\Sigma_\lambda:=B_1\cap H_\lambda^+,
\end{equation}
and for $\delta>0$,
\begin{equation}\label{eq:SigmaDelta}
\Sigma_{\lambda,\delta}:=\{x\in\Sigma_\lambda:\dist(x,\partial\Sigma_\lambda)>\delta\}.
\end{equation}
Moreover, for a measurable set $E$ with $0<|E|<\infty$, we define $$\fint_Ev:=|E|^{-1}\int_Ev.$$

The following weak Harnack theorem plays an important role in proving Theorems \ref{thm:main} and \ref{thm:signchanging}.
\begin{theorem}\label{thm:propagation}
Fix $0<s<1$ and $0\le\lambda_0<1$. For
$0\le\lambda\le\lambda_0,0<\delta\le\frac{1-\lambda_0}{6}$ and $a\ge0,$ let $c,h:\Sigma_\lambda\to\R$ be bounded functions with $M=\sup_{\Sigma_\lambda}c_+$. Suppose $v\in C^{1,1}_{\mathrm{loc}}(\Sigma_\lambda)\cap\mathcal L^{2s}(\R^n)$, satisfying
\begin{equation}\label{eq:super}
(-\Delta)^sv+c(x)v\ge h(x) \qquad\text{in }\Sigma_\lambda,
\end{equation}
and for all $x\in H_\lambda^+$, we have
\begin{equation}\label{eq:shiftanti}
v\ge0, \qquad v(x^\lambda)=2a-v(x).
\end{equation}
 Then there exist $C>0$ and $\beta>0$, depending only on $n,s,\lambda_0,M$, such that
\begin{equation}\label{eq:propestimate}
\sup_{z\in\Sigma_{\lambda,\delta}}\fint_{B_{\delta/2}(z)}v\,dx\le C\delta^{-\beta}\left(\inf_{\Sigma_{\lambda,\delta}}v+\sup_{\Sigma_\lambda}h_-\right).
\end{equation}
\end{theorem}

\begin{proof}
\emph{Step I: Weak Harnack inequality in balls.}

Set $\Omega:=\Sigma_\lambda$ and $E_h:=\sup_\Omega h_-$. Define
\[
\widetilde v(x)=
\begin{cases}
v(x),&x\in H_\lambda^+,\\
v(x)-2a,&x\in H_\lambda^-,\\
0,&x\in T_\lambda.
\end{cases}
\]
Then $\widetilde v(x^\lambda)=-\widetilde v(x)$ and $\widetilde v\ge0$ in $H_\lambda^+$.
For every $x\in\Omega$,
\[
(-\Delta)^s\widetilde v(x)=(-\Delta)^sv(x)+2ac_{n,s}\int_{H_\lambda^-}\frac{dy}{|x-y|^{n+2s}}\ge(-\Delta)^sv(x).
\]
Since $v\ge0$ and $c\le M$ in $\Omega$, it follows that
\[
(-\Delta)^s\widetilde v+M\widetilde v\ge-E_h\quad\text{in }\Omega.
\]
By \cite[Proposition~3.2]{DTV23}, after scaling, for all $B_R(z)\subset\Omega$ with $0<R\leq1$, there holds
\begin{equation}\label{eq:localL1}
\fint_{B_{R/2}(z)}v\,dx\le C_L\left(\inf_{B_{R/2}(z)}v+R^{2s}E_h\right),
\end{equation}
where $C_L>0$ depends only on $n,s,M$.

\emph{Step II: Iteration argument.}

For $e_{1}=(1,0,\cdots,0)\in\mathbb{R}^{n}$, define
\[
z_*:=\frac{1+\lambda}{2}e_1,\qquad d:=\frac{1-\lambda}{2},\qquad Q_*:=B_{d/2}(z_*).
\]
Then we have $$B_d(z_*)\subset\Omega,\qquad d\ge\frac{1-\lambda_0}{2}.$$ 

Next, for a fixed point $z\in\Sigma_{\lambda,\delta}\subset B_{1}$, set
\[
\vartheta:=\frac{1-\lambda_0}{48},\qquad \rho_j:=(1+\vartheta)^j\delta,\qquad 
\]
and $N\in\mathbb Z$ such that $\rho_N\le d$ and $\rho_{N+1}> d$. 

For $j=0,\dots,N$, put
\[
t_j:=\frac{\rho_j-\delta}{d-\delta},\qquad x_j:=(1-t_j)z+t_jz_*, \qquad Q_j:=B_{\rho_j/2}(x_j).
\]
Since $B_\delta(z)$ and $B_d(z_*)$ are contained in the convex set $\Omega$, we have
\begin{equation}\label{eq:convexballs}
B_{\rho_j}(x_j)=(1-t_j)B_\delta(z)+t_jB_d(z_*)\subset\Omega.
\end{equation}
Moreover,
\[
d-\delta\ge\frac{1-\lambda_0}{3},\qquad|z-z_*|\le2.
\]
Thus, for $j<N$,
\[
|x_{j+1}-x_j|=\frac{\vartheta\rho_j}{d-\delta}|z-z_*|\le\frac{\rho_j}{8},
\]
and 
\[
|z_*-x_N|=\frac{d-\rho_N}{d-\delta}|z-z_*|\le\frac{\rho_N}{8}.
\]

Write $r_j:=\rho_j/2$. For $j<N$, $|x_{j+1}-x_j|\le r_j/4$ and $r_{j+1}=(1+\vartheta)r_j$, so
\[
B_{3r_j/4}(x_j)\subset Q_j\cap Q_{j+1}.
\]
For the final link, $|z_*-x_N|\le r_N/4$ and $r_N\le d/2$, so $B_{3r_N/4}(x_N)\subset Q_N\cap Q_*$. Therefore, with
\begin{equation}\label{eq:kappa}
\kappa:=\left(\frac{3}{4(1+\vartheta)}\right)^n,
\end{equation}
every consecutive pair $Q,Q'$ in $Q_0,Q_1,\dots,Q_N,Q_*$ satisfies
\begin{equation}\label{eq:overlap}
|Q\cap Q'|\ge\kappa|Q|,\qquad |Q\cap Q'|\ge\kappa|Q'|.
\end{equation}
The number $L=N+1$ of links satisfies
\begin{equation}\label{eq:numberlinks}
L\le1+\frac{\log(d/\delta)}{\log(1+\vartheta)}.
\end{equation}

For a ball $Q$ in the chain write $A(Q):=\fint_Qv$. Applying \eqref{eq:localL1} to the doubled ball supplied by \eqref{eq:convexballs}, and using $r_{j}\le1$ for all $j$, we obtain
\begin{equation}\label{eq:chainlocal}
A(Q)\le C_L(\inf_Qv+E_h).
\end{equation}
If $Q,Q'$ are consecutive and $E=Q\cap Q'$, then $v\ge0$ and
\eqref{eq:overlap} imply
\[
\inf_Qv\le\fint_Ev\le\frac{|Q'|}{|E|}A(Q')\le\kappa^{-1}A(Q').
\]
Hence,
\begin{equation}\label{eq:link}
A(Q)+E_h\le K_0(A(Q')+E_h),\qquad K_0:=1+C_L(1+\kappa^{-1}),
\end{equation}
and the same inequality holds with $Q$ and $Q'$ interchanged. Define
\begin{equation}\label{eq:alpha}
\alpha:=\frac{\log K_0}{\log(1+\vartheta)}>0.
\end{equation}
Iterating \eqref{eq:link}, using \eqref{eq:numberlinks} and $d\le1$, yields that for any $\zeta\in\Sigma_{\lambda,\delta}$, we have
\begin{align}
A(B_{\delta/2}(z))+E_h&\le C\delta^{-\alpha}(A(Q_*)+E_h),\label{eq:chain1}\\
A(Q_*)+E_h&\le C\delta^{-\alpha}(A(B_{\delta/2}(\zeta))+E_h).\label{eq:chain2}
\end{align}
Let
$m_\delta:=\inf_{\Sigma_{\lambda,\delta}}v$ and choose $\zeta_j\in\Sigma_{\lambda,\delta}$ such that $v(\zeta_j)\to m_\delta$. By \eqref{eq:localL1} with radius $R=\delta$, gives
\[
A(B_{\delta/2}(\zeta_j)) \le C_L(v(\zeta_j)+\delta^{2s}E_h) \le C_L(v(\zeta_j)+E_h).
\]
Use this in \eqref{eq:chain2}, let $j\to\infty$, and substitute the result in \eqref{eq:chain1}. We obtain \eqref{eq:propestimate} with
\begin{equation}\label{eq:beta-effective}
\beta=2\alpha =\frac{2\log K_0}{\log(1+\vartheta)}.
\end{equation}
\end{proof}

\subsection{Antisymmetric maximum principles}
We use the notation $H_\lambda^-$, $T_\lambda$ and $x^\lambda$ from Section~2.1. The two principles below are standard in nonlocal moving planes. Narrow-region principles for antisymmetric functions are from \cite[Theorem~2.3]{CLL17}. Small-volume versions for general kernels were established  in \cite{JW16}. Quantitative adaptations for Serrin-type problems are in \cite[Proposition~3.1]{DPTV23Serrin}. We give a quantitative version for both  maximum principles because the scales $\delta^{2s}$ and $|D|^{2s/n}$ matter later. The short proof makes constants and signs explicit. The $\{x_1>\lambda\}$ version follows by reflection.

\begin{proposition}[Antisymmetric narrow-region and small-volume estimates]\label{prop:maxprinciples}
For $0<s<1$, let $D\subset H_\lambda^-$ be a bounded open set and  $c,f:D\to\R$ be bounded functions. Suppose $w\in\mathcal L^{2s}(\R^n)\cap C^{1,1}_{\mathrm{loc}}(D)$ and is an upper semicontinuous function on $\overline D$ satisfying
\begin{equation}\label{eq:antiMPass}
\begin{cases}
(-\Delta)^sw+c(x)w\le f(x),&\text{in } D,\\
w(x^\lambda)=-w(x),&\text{in } H_\lambda^-,\\
%w=0,&x\in T_\lambda,\\
w\le0,&\text{in } \overline{H_\lambda^-}\setminus D.
\end{cases}
\end{equation}
Define $c_0:=\sup\limits_{D} c_-$ and $F_0:=\sup\limits_{D} f_+.$
\begin{enumerate}[label=\textup{(\roman*)}]
\item There exist $\delta_0>0$ and $C>0$, depending only on $n,s,c_0$ such that if $D\subset\{x\in H_\lambda^-:\lambda-\delta<x_1<\lambda\},$ then for $0<\delta\le\delta_0$,
\begin{equation}\label{eq:narrow}
\sup_Dw_+\le C\delta^{2s}F_0.
\end{equation}
%If $c_0=0$, no restriction on $\delta$ is needed.

\item There exist $m_0>0$ and $C>0$, depending only on $n,s,c_0$ such that, if $|D|\le m_0$,
\begin{equation}\label{eq:smallvolume}
\sup_Dw_+\le C|D|^{2s/n}F_0.
\end{equation}
%If $c_0=0$, no smallness restriction on $|D|$ is needed.
\end{enumerate}
\end{proposition}
\begin{remark}
    If $c_0=0$, then for (i), no restriction on $\delta$ is needed, and for (ii), no smallness restriction on $|D|$ is needed.
\end{remark}
\begin{proof}
Without loss of generality, we assume $M_0:=\max_{\overline D}w_+>0$. By upper semicontinuity and the last condition in \eqref{eq:antiMPass}, we get that there exists a point $x_0\in D$ such that $w(x_0)=M_0$. Moreover $w(y)\le M_0$ for every $y\in H_\lambda^-$. Reflecting the complementary half-space into $H_\lambda^-$ gives
\begin{equation}
\begin{aligned}
(-\Delta)^sw(x_0)=c_{n,s}\int_{\R^{n}}&\frac{M_{0}-w(y)}{|x_0-y|^{n+2s}}dy,\\
=c_{n,s}\int_{H_\lambda^-}&\Bigg[(M_0-w(y))\left(\frac1{|x_0-y|^{n+2s}}-\frac1{|x_0-y^\lambda|^{n+2s}}\right) \\
&+\frac{2M_0}{|x_0-y^\lambda|^{n+2s}}\Bigg]dy.\label{eq:MPpair}
\end{aligned}
\end{equation}
The integrals in \eqref{eq:MPpair} are absolutely convergent. Moreover, since for $x_0,y\in H_\lambda^-$ one has $|x_0-y^\lambda|>|x_0-y|$, so both terms in the integral are nonnegative.

For part (i), we obtain
\begin{equation}\label{eq:x_0}
    (-\Delta)^sw(x_0) \ge2c_{n,s}M_0\int_{H_\lambda^-}\frac{dy}{|x_0-y^\lambda|^{n+2s}}\ge\zeta_{n,s}\delta^{-2s}M_0,
\end{equation}
where
\[
\zeta_{n,s}:=2c_{n,s}\int_{\{z_1>1\}}|z|^{-n-2s}\,dz>0.
\]
Combining \eqref{eq:antiMPass} and \eqref{eq:x_0}, we get
\begin{equation}\label{eq:F_0}
    F_0\ge f(x_0) \ge(\zeta_{n,s}\delta^{-2s}-c_0)M_0.
\end{equation}
If $c_0>0$, choose $\delta_0$ so that $c_0\delta_0^{2s}\le\zeta_{n,s}/2$. This proves \eqref{eq:narrow}. If $c_0=0$, \eqref{eq:F_0} directly gives \eqref{eq:narrow} for every $\delta>0$.

For part (ii), $w\le0$ in $H_\lambda^-\setminus D$, so $M_0-w(y)\ge M_0$ there. Hence, using $M_0:=\max_{\overline D}w_+>0$ and $|x_0-y^\lambda|>|x_0-y|$ for all $x_0,y\in H_\lambda^-$, we obtain
\begin{equation}
\begin{aligned}
(-\Delta)^sw(x_0)
&\ge c_{n,s}\int_{H_\lambda^-\backslash D}M_{0}\left(\frac1{|x_0-y|^{n+2s}}-\frac1{|x_0-y^\lambda|^{n+2s}}\right)dy\\
&\quad+c_{n,s}\int_{H_{\lambda}^{-}}\frac{2M_0}{|x_0-y^\lambda|^{n+2s}}dy,\\
%&\ge c_{n,s}\int_{H_\lambda^-\backslash D}M_{0}\left(\frac1{|x_0-y|^{n+2s}}+\frac1{|x_0-y^\lambda|^{n+2s}}\right)dy,\\
&\ge c_{n,s}M_0\int_{H_\lambda^-\setminus D}\frac{dy}{|x_0-y|^{n+2s}}.
\end{aligned}
\end{equation}
Let $\omega_n=|B_1|$ and put
\[
r:=\left(\frac{4|D|}{\omega_n}\right)^{1/n}.
\]
Because $x_0\in H_\lambda^-$, 
\[
|H_\lambda^-\cap B_r(x_0)|\ge\frac12\omega_nr^n=2|D|,
\]
so
\[
|(H_\lambda^-\setminus D)\cap B_r(x_0)|\ge|D|.
\]
Consequently
\begin{equation}\label{eq:int-loss}
    \int_{H_\lambda^-\setminus D}\frac{dy}{|x_0-y|^{n+2s}}\ge r^{-n-2s}|D|=c_n|D|^{-2s/n}.
\end{equation}
Combining \eqref{eq:antiMPass} with \eqref{eq:int-loss}, we get
\[
F_0\ge(c|D|^{-2s/n}-c_0)M_0.
\]
If $c_0>0$, choose $m_0>0$ so that $c_0\le(c/2)m_0^{-2s/n}$. This gives \eqref{eq:smallvolume}. If $c_0=0$, no restriction on $|D|$ is needed.
\end{proof}
\subsection{Boundary regularity and the Green function}
We shall use the following global estimate of Ros-Oton and Serra \cite{ROS14} and Green representation \cite{Bucur16}.

\begin{proposition}\label{prop:boundaryreg}
Let $\Omega\subset\R^n$ be a bounded $C^{1,1}$ domain and
$q\in L^\infty(\Omega)$. If $u$ is a classical solution of
\[
(-\Delta)^su=q\quad\text{in }\Omega,\qquad u=0\quad\text{a.e.\ in }\R^n\setminus\Omega,
\]
then $u$ has a representative in $C^s(\R^n)$ and
\[
\|u\|_{C^s(\R^n)}\le C\|q\|_{L^\infty(\Omega)},
\]
where $C$ depends only on $n,s,\Omega$.
\end{proposition}

Let 
\begin{equation}\label{eq:GreenR}
G_{R}(x,y)=\kappa_{n,s}|x-y|^{2s-n}\int_0^{\rho_{R}(x,y)}\frac{t^{s-1}}{(1+t)^{n/2}}\,dt,
\end{equation}
where
\begin{equation}
\rho_R(x,y)=\frac{(R^2-|x|^2)(R^2-|y|^2)}{R^2|x-y|^2},\qquad \kappa_{n,s}=\frac{\Gamma(n/2)}{2^{2s}\pi^{n/2}\Gamma(s)^2}.
\end{equation} 

Define the bilinear form
\[
\E(v,w)=\frac{C_{n,s}}2\iint_{\R^n\times\R^n}\frac{(v(x)-v(y))(w(x)-w(y))}{|x-y|^{n+2s}}\,dx\,dy.
\]
A weak solution is a function $u\in X_0^s(B_R)$ satisfying for every $\varphi\in X_0^s(B_R)$,
\begin{equation}\label{eq:weak}
\E(u,\varphi)=\int_{B_R}h\varphi\, dx.
\end{equation}
Then we obtain the Green representation for a weak solution in a ball. The proof of the proposition is similar to that of \cite[Theorem~3.2]{Bucur16}.
\begin{proposition}\label{prop:green}
For $n>2s$, let $h\in L^\infty(B_R)$ and define
\begin{equation}\label{eq:green}
u(x)=
\begin{cases}
\displaystyle\int_{B_R}G_R(x,y)h(y)\, d y,&x\in B_R,\\[1mm]
0,&x\in\R^n\setminus B_R.
\end{cases}
\end{equation}
Then $u$ is the unique weak solution of
\begin{equation}\label{eq:problem}
\begin{cases}
(-\Delta)^s u=h&\text{in }B_R,\\
u=0&\text{in }\R^n\setminus B_R.
\end{cases}
\end{equation}
\end{proposition}

\section{Approximate symmetry}

We first establish the lower boundary estimate used in the proof for Theorem \ref{thm:main}.
\begin{lemma}\label{lem:lower}
Let $n\ge2$ and $0<s<1$. Suppose $u\in X_{0}^{s}(B_{1})\cap C^{1,1}_{loc}(B_{1})$ satisfies
\[
(-\Delta)^su=q(x)\quad\text{in }B_1,\qquad u=0\quad\text{in }\R^n\setminus B_1,
\]
where $q\in L^\infty(B_1)$ and $q\ge0$. Assume
\begin{equation}\label{eq:u neq 0}
    \frac1{C_0}\le\|u\|_{L^\infty(B_1)}\le C_0.
\end{equation}
Then there exists $c_*>0$, depending only on $n,s,C_0$ and $\|q\|_{L^\infty(B_1)}$, such that for all $x\in B_1$,
\begin{equation}\label{eq:lowerbd}
u(x)\ge c_*(1-|x|)^s.
\end{equation}
\end{lemma}
\begin{proof}[Proof of Lemma \ref{lem:lower}]
Put $Q:=\|q\|_{L^\infty(B_1)}$. Without loss of generality, we assume $Q>0$. Otherwise, the strong maximum principle would give $u\equiv0$, which contradicts \eqref{eq:u neq 0}. By Proposition~\ref{prop:boundaryreg},
\begin{equation}\label{eq:holderu}
\|u\|_{C^s(\R^n)}\le C_*:=C(n,s)Q.
\end{equation}
Choose $x_0\in B_1$ with $u(x_0)\ge(2C_0)^{-1}$. If $z_0\in\partial B_1$ satisfies $|x_0-z_0|=1-|x_0|$, then $u(z_0)=0$, and \eqref{eq:holderu} yields
\[
\frac1{2C_0}\leq u(x_{0})-u(z_{0})\le C_*(1-|x_0|)^s.
\]
Hence,
\begin{equation}\label{eq:d0}
d_B(x_0):=1-|x_0|\ge d_0:=\min\left\{\frac12,\left(\frac1{2C_0C_*}\right)^{1/s}\right\}>0.
\end{equation}

Let $G(x,y)=G_{1}(x,y)$ be the Green function in \eqref{eq:GreenR}. Then, by 
\begin{equation}\label{eq:Green}
G(x,y)=\kappa_{n,s}|x-y|^{2s-n}\int_0^{\rho(x,y)}\frac{t^{s-1}}{(1+t)^{n/2}}\,dt,\qquad\rho(x,y)=\frac{(1-|x|^2)(1-|y|^2)}{|x-y|^2}.
\end{equation}
We have three consequences of \eqref{eq:Green}.

First, there is $c=c(n,s)>0$ such that for all $x,y\in B_1$,
\begin{equation}\label{eq:Greenlower}
G(x,y)\ge c\,d_B(x)^s\,d_B(y)^s,
\end{equation}
Indeed, put
$A=(1-|x|^2)(1-|y|^2),$ and thus, $\rho=A/|x-y|^2.$
If $\rho\le1$, then 
\[
G(x,y)\ge c|x-y|^{2s-n}\int_0^{\rho(x,y)}t^{s-1}\,dt=c|x-y|^{2s-n}\rho^{s}= cA^s|x-y|^{-n}\ge c\,d_B(x)^s\,d_B(y)^s,
\]
because $|x-y|\le2$. If $\rho>1$,
\[
G(x,y)\ge c|x-y|^{2s-n}\int_0^{1}\frac{t^{s-1}}{(1+t)^{n/2}}\,dt=c|x-y|^{2s-n}\ge cd_B(x)^s\,d_B(y)^s.
\]
since $d_B(x)^s\,d_B(y)^s\le1$ and $|x-y|^{2s}\le2^{2s}$.

Second, for $0<r\le d_0/4$,
\begin{equation}\label{eq:Greenlocal}
\int_{B_r(x_0)}G(x_0,y)\,dy\le Cr^{2s}=:\omega_0(r).
\end{equation}
since
\begin{equation}\label{eq:G-2}
\begin{aligned}
    G(x_0,y)\le C|x-y|^{2s-n}\int_0^{+\infty}\frac{t^{s-1}}{(1+t)^{n/2}}\,dt\le C|x_0-y|^{2s-n}.
\end{aligned}  
\end{equation}
The integral in \eqref{eq:G-2} converges by $n>2s$.

Third, for every fixed $0<r\le d_0/4$ there is $C=C(n,s,d_0,r)$ such that for all $|y-x_0|\ge r$,
\begin{equation}\label{eq:Greenupper}
G(x_0,y)\le C\,d_B(y)^s.
\end{equation}
To see this, if $\rho(x_0,y)\le1$, then
\begin{equation}\label{eq:G-2A}
    G(x_{0},y)\le C|x_{0}-y|^{2s-n}\int_0^{\rho(x_{0},y)}t^{s-1}\,dt\le C\rho^{s}\le Cd_{B}(y)^{s}.
\end{equation}
If $\rho(x_0,y)>1$, then
\[
(1-|x_0|^2)(1-|y|^2)>|x_0-y|^2\ge r^2,
\]
so $d_B(y)\ge r^2/2$. Hence,
\begin{equation}\label{eq:G-2B}
    G(x_0,y)\le C|x_{0}-y|^{2s-n}\int_0^{+\infty}\frac{t^{s-1}}{(1+t)^{n/2}}\,dt\le C\leq Cd_{B}(y)^{s}.
\end{equation}
Then \eqref{eq:G-2A} and \eqref{eq:G-2B} implies \eqref{eq:Greenupper}.

By Proposition \ref{prop:green},
\begin{equation}\label{eq:Greenrep}
u(x)=\int_{B_1}G(x,y)q(y)\,dy.
\end{equation}
Choose $r\le d_0/4$, depending only on the data such that
\[
Q\omega_0(r)\le\frac1{4C_0}.
\]
Then \eqref{eq:Greenlocal}, \eqref{eq:Greenrep} and $u(x_0)\ge(2C_0)^{-1}$ give
\[
\int_{B_1\setminus B_r(x_0)}G(x_0,y)q(y)\,dy\ge u(x_0)-\omega_{0}(r)Q\geq\frac1{4C_0}.
\]
Using \eqref{eq:Greenupper} and $q\ge0$, we obtain
\begin{equation}\label{eq:moment}
\int_{B_1}d_B(y)^s q(y)\,dy\ge\int_{B_1\setminus B_r(x_0)}d_B(y)^s q(y)\,dy \ge C\int_{B_1\setminus B_r(x_0)}G(x_0,y)q(y)\,dy\ge c_0,
\end{equation}
with $c_0>0$ depending only on the data.

Finally, \eqref{eq:Greenlower}, \eqref{eq:Greenrep} and \eqref{eq:moment} yield
\[
u(x)\ge c\,d_B(x)^s\int_{B_1}d_B(y)^s q(y)\,dy\ge c_*\,d_B(x)^s.
\]
This proves \eqref{eq:lowerbd}.
\end{proof}

\subsection{Proof of Theorem~\ref{thm:main}}
\begin{proof}[Proof of Theorem~\ref{thm:main}]
Without loss of generality, we first move the planes orthogonal to $e_1$. For $\lambda\in(0,1)$, we use the notation $\Sigma_\lambda$, $T_\lambda$ and $x^\lambda$ defined in Section~2.1. Also, set
\[
u_\lambda(x):=u(x^\lambda),\qquad w_\lambda:=u-u_\lambda,\qquad \varepsilon:=\mathcal D(k).
\]
%If $\varepsilon=0$, then $k$ is radial and nonincreasing. The qualitativesymmetry result discussed in the Introduction gives that $u$ is radial andnonincreasing, hence $\mathcal D(u)=0$ and there is nothing to prove.
If $\varepsilon=0$, then $k$ is radial and nonincreasing, we can obtain $u$ is radial and nonincreasing by \cite[Theorem 1.1]{Zhang15}, hence $\mathcal D(u)=0$. We may therefore assume $\varepsilon>0$. There is also no loss in imposing
\[
0<\varepsilon\le\varepsilon_0\le1,
\]
where $\varepsilon_0>0$ to be determined later.
Indeed, $\mathcal D(u)\le2C_0$, so the case $\varepsilon\ge\varepsilon_0$ is absorbed by enlarging the final constant in \eqref{eq:mainest}.

Define
\begin{equation}\label{eq:clambda}
c_\lambda(x):=
\begin{cases}
-k(x)\dfrac{g(u(x))-g(u_\lambda(x))}{u(x)-u_\lambda(x)},&u(x)\ne u_\lambda(x),\\[8pt]
0,&u(x)=u_\lambda(x).
\end{cases}
\end{equation}
Then
\begin{equation}\label{eq:cBound}
\|c_\lambda\|_{L^\infty(\Sigma_\lambda)}\le \|k\|_{L^\infty(B_1)}\|g\|_{C^{0,1}([0,C_0])}=:M_0.
\end{equation}
For all $x\in\Sigma_\lambda$, we have
\begin{equation}\label{eq:u_lambda}
    (-\Delta)^{s}u_{\lambda}(x)=k_{\lambda}(x)g(u_{\lambda}(x)),
\end{equation}
where $k_\lambda(x):=k(x^\lambda)$.

Combining \eqref{eq:main} and \eqref{eq:u_lambda}, for all $x\in\Sigma_\lambda$, we get
\begin{equation}\label{eq:wEq}
(-\Delta)^sw_\lambda+c_\lambda(x)w_\lambda=(k-k_\lambda)g(u_\lambda)\qquad\text{in }\Sigma_\lambda.
\end{equation}
Let $r=|x|$, $\rho=|x^\lambda|$. If $\rho>0$, write $e=x/r$ and $e_\lambda=x^\lambda/\rho$. Since $|x^\lambda|\le|x|$, we have
\[
k(x)-k(x^\lambda)=\bigl(k(re)-k(\rho e)\bigr)+\bigl(k(\rho e)-k(\rho e_\lambda)\bigr).
\]
Hence,
\begin{equation}\label{eq:kdiff}
(k(x)-k(x^\lambda))_+\le \bigl(k(re)-k(\rho e)\bigr)_++\osc_{\partial B_\rho}k\le\varepsilon.
\end{equation}
The case $\rho=0$ follows by continuity. Since $g\ge0$ and $0\le u_\lambda\le C_0$,
\begin{equation}\label{eq:forcing}
\|((k-k_\lambda)g(u_\lambda))_+\|_{L^\infty(\Sigma_\lambda)}\le C\varepsilon.
\end{equation}

For $\mu$ close to $1$, the cap $\Sigma_\mu$ lies in the strip $\mu<x_1<1$ of width $1-\mu$. The function $w_\mu$ is antisymmetric with respect to $T_\mu$, and in the right half-space outside $\Sigma_\mu$ we have $w_\mu\le0$, since $u=0$ outside $B_1$ and $u\ge0$ in $B_1$. Let $\lambda_1$ be fixed such that $\lambda_1\ge \max\{1-\delta_0,1/2\}$, where $\delta_0$ is the constant determined in Proposition~\ref{prop:maxprinciples}. By Proposition~\ref{prop:maxprinciples}(i), \eqref{eq:cBound} and \eqref{eq:forcing}, we obtain
\[
\|w_\mu^+\|_{L^\infty(\Sigma_\mu)}\le C^{*}_{1}\varepsilon
\]
for every $\mu\in[\lambda_1,1)$.

Define
\begin{equation}\label{eq:Lambda}
\Lambda:=\left\{\lambda\in(0,1):\|w_\mu^+\|_{L^\infty(\Sigma_\mu)}\le A\varepsilon\ \text{for every }\mu\in[\lambda,1)\right\},
\end{equation}
where $A=\max\{C_1^*,C_2^*\}$. Here, $C_2^*$ is used in \eqref{eq:A-2}.

Then
\[
[\lambda_1,1)\subset\Lambda.
\]

%\smallskip
%\noindent\emph{Step I: $\Lambda\ne\varnothing$.}
Set
\begin{equation}\label{eq:lstar}
\lambda_*:=\inf\Lambda.
\end{equation}
Because $\Lambda$ is closed upward, the estimate in \eqref{eq:Lambda} holds for every $\mu>\lambda_*$. Letting $\mu\downarrow\lambda_*$ and using the global continuity of $u$ gives
\begin{equation}\label{eq:atstar}
w_{\lambda_*}\le A\varepsilon\qquad\text{in }\Sigma_{\lambda_*}.
\end{equation}

\smallskip
\noindent\emph{Step I: $\lambda_*\le1/4$.}
Assume for the contradiction that
\begin{equation}\label{eq:starbig}
\frac14<\lambda_*\le\lambda_1.
\end{equation}
Define
\begin{equation}\label{eq:vstar}
v:=u_{\lambda_*}-u+A\varepsilon.
\end{equation}
By \eqref{eq:atstar}, $v\ge0$ in $\Sigma_{\lambda_*}$. If $x\in H_{\lambda_*}^+\backslash B_1$, then $u(x)=0$ and $u(x^{\lambda_*})\ge0$, so $v(x)\ge A\varepsilon$. Thus,
\begin{equation}\label{eq:vhalf}
v\ge0\qquad\text{in }H_{\lambda_*}^+.
\end{equation}
Since $u_{\lambda_*}-u$ is antisymmetric, for all $x\in H_{\lambda_*}^{+}$, we have
\begin{equation}\label{eq:vshift}
v(x^{\lambda_*})=2A\varepsilon-v(x).
\end{equation}
From \eqref{eq:wEq}, using $w_{\lambda_*}=A\varepsilon-v$, we obtain
\begin{align*}
(-\Delta)^sv+c_{\lambda_*}^+v
&=(k_{\lambda_*}-k)g(u_{\lambda_*})+A\varepsilon c_{\lambda_*}+c_{\lambda_*}^-v\\
&\ge -(k-k_{\lambda_*})_+g(u_{\lambda_*})-A\varepsilon c_{\lambda_*}^-\\
&\ge-C\varepsilon\qquad\text{in }\Sigma_{\lambda_*},
\end{align*}
where the last inequality follows from \eqref{eq:cBound} and \eqref{eq:kdiff}.

Apply Lemma~\ref{lem:lower} to $q(x)=k(x)g(u(x))$. Here $q\ge0$ and $\|q\|_\infty\le\|k\|_\infty\|g\|_{L^\infty([0,C_0])}$, so the lemma gives $c_*>0$ such that  for all $x\in B_1$,
\begin{equation}\label{eq:lowerMain}
u(x)\ge c_*(1-|x|)^s.
\end{equation}
Proposition~\ref{prop:boundaryreg} gives that there exists a positive constant $C_H$ such that for all $x\in B_1$,
\begin{equation}\label{eq:upperMain}
u(x)\le C_H(1-|x|)^s,\qquad \|u\|_{C^{s}(\R^{n})}\le C_{H}
\end{equation}

Choose a fixed $\eta>0$, depending only on the data, such that
\begin{equation}\label{eq:etaFixed}
\eta<\min\left\{\frac{1-\lambda_1}{6},\frac1{20},\frac{m_{0}}{2C_{n}}\right\}, \qquad C_H\left(\frac{5\eta}{2}\right)^s \le\frac{c_*}{2}\left(\frac{1-\lambda_1}{4}\right)^s,
\end{equation}
where $C_n$ is the geometric constant used in \eqref{eq:Dmeasure} and $m_{0}$ is the constant determined in Proposition \ref{prop:maxprinciples}. Let
\[
z_\eta:=(1-2\eta)e_1.
\]
Since $1-2\eta-\lambda_*\ge1-2\eta-\lambda_1>\eta$ and $1-|z_\eta|=2\eta$, we have $z_\eta\in\Sigma_{\lambda_*,\eta}$. If $x\in B_{\eta/2}(z_\eta)$, then $x_1\ge1-5\eta/2$, and
\begin{align*}
1-|x^{\lambda_*}|^2=1-|x|^2+4\lambda_*(x_1-\lambda_*)\ge4\lambda_*\left(1-\frac{5\eta}{2}-\lambda_*\right)\ge\frac{1-\lambda_1}{2},
\end{align*}
where we used $\lambda_*\ge1/4$, $\lambda_*\le\lambda_1$, and $5\eta/2<(1-\lambda_1)/2$. Hence,
\begin{equation}\label{eq:refdistfixed}
1-|x^{\lambda_*}|\ge\frac{1-\lambda_1}{4}.
\end{equation}
Also $1-|x|\le5\eta/2$. By (\ref{eq:lowerMain})-(\ref{eq:refdistfixed}), for all $x\in B_{\eta/2}(z_\eta)$,
\begin{equation}\label{eq:positiveballfixed}
u(x^{\lambda_*})-u(x)\ge m_1:=\frac{c_*}{2}\left(\frac{1-\lambda_1}{4}\right)^s>0.
\end{equation}
Consequently,
\[
\fint_{B_{\eta/2}(z_\eta)}v\,dx\ge m_1.
\]
Theorem~\ref{thm:propagation} with $\delta=\eta$ yields
\[
m_1\le C\eta^{-\beta}\left(\inf_{\Sigma_{\lambda_*,\eta}}v+C\varepsilon\right).
\]
Since $v=(u_{\lambda_*}-u)+A\varepsilon$, there exist constants $c_2,C_2>0$, depending only on the data, such that
\begin{equation}\label{eq:corefixed}
\inf_{\Sigma_{\lambda_*,\eta}}(u_{\lambda_*}-u)\ge c_2\eta^\beta-C_2\varepsilon\geq\frac{c_2}{2}\eta^\beta:=q_{0},
\end{equation}
where we take $\epsilon_{0}$ to satisfy
\begin{equation}\label{eq:epsilon0}
    C_2\varepsilon_{0}\le\frac{c_2}{2}\eta^\beta.
\end{equation}

Choose $\sigma_0>0$ so that
\begin{equation}\label{eq:sigmafixed}
C_H(2\sigma_0)^s\le\frac{q_0}{2},\qquad\sigma_0<\min\{\eta,\lambda_*/2\}.
\end{equation}
For $x\in\Sigma_{\lambda_*,\eta}$ and $\mu\in[\lambda_*-\sigma_0,\lambda_*]$,
\[
|x^\mu-x^{\lambda_*}|=2|\mu-\lambda_*|.
\]
Hence \eqref{eq:upperMain}, \eqref{eq:corefixed} and \eqref{eq:sigmafixed} give
\begin{equation}\label{eq:coremu}
u_\mu-u\ge (u_{\lambda_*}-u)-C_H(2\sigma_0)^s\ge\frac{q_0}{2}\qquad\text{in }\Sigma_{\lambda_*,\eta}.
\end{equation}
Set
\[
D_\mu:=\Sigma_\mu\setminus\overline{\Sigma_{\lambda_*,\eta}}.
\]
If $x\in D_\mu$, then either its distance from $T_{\lambda_*}$ is at most $\eta$ (which gives $\mu<x_1\le\lambda_*+\eta$), or its distance from the spherical part of $\partial\Sigma_{\lambda_*}$ is at most $\eta$ (which implies $1-|x|\le\eta$). Hence,
\begin{equation}\label{eq:Dmeasure}
|D_\mu|\le C_n(\eta+\sigma_0)\le 2C_{n}\eta\le m_{0}.
\end{equation}
By \eqref{eq:wEq}, \eqref{eq:Dmeasure}, $w_{\mu}(x)=-w_{\mu}(x^\mu)$, and Proposition~\ref{prop:maxprinciples}(ii), we obtain for all $\lambda_*-\sigma_0\le\mu\le\lambda_*$,
\begin{equation}\label{eq:A-2}
    \|w_\mu^+\|_{L^\infty(D_\mu)}\le C_2^*\varepsilon\le A\varepsilon.
\end{equation}
Combining \eqref{eq:coremu} and \eqref{eq:A-2}, we get
\begin{equation}
    \|w_\mu^+\|_{L^\infty(\Sigma_\mu)}\le A\varepsilon.
\end{equation}
Together with the definition of $\lambda_*$ for $\mu>\lambda_*$, this implies $\lambda_*-\sigma_0\in\Lambda$, contradicting \eqref{eq:lstar}. Hence,
\begin{equation}\label{eq:starquarter}
\lambda_*\le\frac14.
\end{equation}

\smallskip
\noindent\emph{Step II: quantitative bound for $\lambda_*$.}
If $\lambda_*=0$, there is nothing to prove. Assume $\lambda_*>0$. Let $\beta$ be the exponent in Theorem~\ref{thm:propagation}, used with $\lambda_0=\lambda_1$. We prove
\begin{equation}\label{eq:starquant}
\lambda_*\le C\varepsilon^{1/(s+\beta)}.
\end{equation}
Choose $\theta_0>0$, depending only on the data, such that
\begin{equation}\label{eq:theta}
\theta_0\le\min\left\{\frac{1-\lambda_1}{6},\left(\frac{c_*}{2C_H(5/2)^s}\right)^{1/s}, \frac{2m_{0}}{C_{n}}\right\},
\end{equation}
where $C_{n}$ is used in \eqref{eq:Pmu}. Put
\begin{equation}\label{eq:etaScale}
\eta:=\theta_0\lambda_*.
\end{equation}
By \eqref{eq:starquarter}, $\lambda_*\le1/4$. Recall $z_\eta=(1-2\eta)e_1$. Since
\[
1-2\eta-\lambda_*\ge1-2\times\frac16\times\frac14-\frac14=\frac23>\eta,
\]
so $z_\eta\in\Sigma_{\lambda_*,\eta}$. If $x\in B_{\eta/2}(z_\eta)$, then
\begin{equation}\label{eq:z-eta}
    x_1\ge1-\frac{5\eta}{2}, \qquad 1-|x|\le\frac{5\eta}{2}.
\end{equation}
Since $\eta=\theta_0\lambda_*$, $\theta_0\le1/4$, and $\lambda_*\le1/4$,
\[
x_1-\lambda_*\ge1-\left(1+\frac{5\theta_0}{2}\right)\lambda_*\ge1-\frac{13}{8}\cdot\frac14=\frac{19}{32}>\frac12.
\]
Hence,
\[
2(1-|x^{\lambda_*}|)\geq(1-|x^{\lambda_*}|)(1+|x^{\lambda_*}|)=1-|x^{\lambda_*}|^2=1-|x|^2+4\lambda_*(x_1-\lambda_*)\ge2\lambda_*.
\]
Therefore,
\begin{equation}\label{eq:refdistscale}
1-|x^{\lambda_*}|\ge\lambda_*.
\end{equation}
At the same time, $1-|x|\le5\eta/2$. By \eqref{eq:lowerMain}, \eqref{eq:upperMain}, \eqref{eq:theta}, \eqref{eq:etaScale} and \eqref{eq:refdistscale}, for all $x\in B_{\eta/2}(z_\eta)$,
\begin{equation}\label{eq:positiveballscale}
u(x^{\lambda_*})-u(x)\ge c\lambda_*^s.
\end{equation}

Applying Theorem~\ref{thm:propagation} to the function $v$ in \eqref{eq:vstar} and using \eqref{eq:positiveballscale}, we obtain
\begin{equation}\label{eq:coreScale}
\inf_{\Sigma_{\lambda_*,\eta}}(u_{\lambda_*}-u)\ge c\eta^\beta\lambda_*^s-C\varepsilon\ge c_{3}\lambda_*^{s+\beta}-C_{3}\varepsilon.
\end{equation}
Suppose, contrary to the desired estimate, that
\begin{equation}\label{eq:contrscale}
c_{3}\lambda_*^{s+\beta}>2C_{3}\varepsilon.
\end{equation}
Then \eqref{eq:coreScale} yields
\begin{equation}\label{eq:qstar}
\inf_{\Sigma_{\lambda_*,\eta}}(u_{\lambda_*}-u)\ge q_*:=c\lambda_*^{s+\beta}>0.
\end{equation}
Choose $\sigma_0>0$ such that
\begin{equation}\label{eq:sigma00}
C_H(2\sigma_0)^s\le\frac{q_*}{2},\qquad\sigma_0<\min\{\eta,\lambda_*/2\}.
\end{equation}
For $x\in\Sigma_{\lambda_*,\eta}$ and $\mu\in[\lambda_*-\sigma_0,\lambda_*]$, the global $C^s$ estimate gives
\begin{equation}\label{eq:u-second}
    u_\mu(x)-u(x)\ge u_{\lambda_*}(x)-u(x)-C_H(2\sigma_0)^s\ge\frac{q_*}{2}.
\end{equation}
By \eqref{eq:Dmeasure}, \eqref{eq:starquarter}, \eqref{eq:etaScale}, and \eqref{eq:sigma00},
\begin{equation}\label{eq:Pmu}
|D_\mu|\le2C_n\eta=2C_n\theta_0\lambda_*\le\frac{C_n\theta_0}{2}\le m_{0}.
\end{equation}
Applying Proposition~\ref{prop:maxprinciples}(ii) to $w_\mu$ on $D_\mu$, with \eqref{eq:u-second}, gives for all $\lambda_*-\sigma_0\le\mu\le\lambda_*$,
\begin{equation}\label{A-3}
 \|w_\mu^+\|_{L^\infty(\Sigma_\mu)}\le C_2^*\varepsilon\le A\varepsilon.   
\end{equation}
Together with the defining estimate for every $\mu>\lambda_*$, this yields $\lambda_*-\sigma_0\in\Lambda$, contradicting \eqref{eq:lstar}.
Thus, \eqref{eq:contrscale} is impossible, and \eqref{eq:starquant} follows.

\smallskip
\noindent\emph{Step III: reflection through the origin and radial monotonicity.}
Set
\begin{equation}\label{eq:lambda2}
\lambda_2:=C\varepsilon^{1/(s+\beta)},
\end{equation}
with $C$ large enough that $\lambda_*\le\lambda_2$. If $\lambda_2\ge1$, the desired estimate is already absorbed by the uniform bound on $u$, so we may assume $\lambda_2<1$. By the definition of $\lambda_*$ and the continuity in the plane parameter,
\begin{equation}\label{eq:wlambda2}
\|(u-u_{\lambda_2})_+\|_{L^\infty(\Sigma_{\lambda_2})}\le A\varepsilon.
\end{equation}
If $x=(x_1,x')\in B_1$ and $x_1>\lambda_2$, then
\begin{align*}
u(x_1,x')-u(-x_1,x')={}&\bigl[u(x_1,x')-u(2\lambda_2-x_1,x')\bigr]+\bigl[u(2\lambda_2-x_1,x')-u(-x_1,x')\bigr]\\
\le{}&A\varepsilon+C_H(2\lambda_2)^s\le{}C\varepsilon^{s/(s+\beta)}.
\end{align*}
If $0<x_1\le\lambda_2$, the global $C^s$ estimate gives directly
\[
u(x_1,x')-u(-x_1,x')\le C_H(2x_1)^s\le C\lambda_2^s\le C\varepsilon^{s/(s+\beta)}.
\]
Running the moving plane from the opposite side gives the reverse inequality. Consequently,
\begin{equation}\label{eq:originreflection}
|u(x_1,x')-u(-x_1,x')|\le C\varepsilon^\gamma,\qquad\gamma:=\frac{s}{s+\beta}\in(0,1),
\end{equation}
The argument is rotation invariant, so \eqref{eq:originreflection} holds for reflection across every hyperplane through the origin. If $|x|=|y|$, there exists a reflection sending $x$ to $y$. Hence,
\begin{equation}\label{eq:sphosc}
\sup_{0<r<1}\osc_{\partial B_r}u\le C\varepsilon^\gamma.
\end{equation}

It remains to control the radial monotonicity. Fix $e\in\Sn$ and $0\le\rho<r<1$, and move the planes orthogonal to $e$. Let
\[\lambda:=\frac{r+\rho}{2}.\]
The reflection of $re$ across $\{x\cdot e=\lambda\}$ is $\rho e$.

If $\lambda\ge\lambda_2$, then $\lambda\ge\lambda_*(e)$. Since $\lambda\in\Lambda$ and $0\le\varepsilon\le1$, we have
\begin{equation}\label{m-1}
    u(re)-u(\rho e)\le A\varepsilon\le A\epsilon^{\gamma}.
\end{equation}
If $\lambda<\lambda_2$, then $r<2\lambda_2$ and
\begin{equation}\label{m-2}
    u(re)-u(\rho e)\le C_H(r-\rho)^s\le C\lambda_2^s\le C\varepsilon^\gamma.
\end{equation}
Combining \eqref{m-1} and \eqref{m-2}, we conclude
\begin{equation}\label{eq:radmono}
\sup_{e\in\Sn}\sup_{0\le\rho<r<1}\bigl(u(re)-u(\rho e)\bigr)_+\le C\varepsilon^\gamma.
\end{equation}
Combining \eqref{eq:sphosc} and \eqref{eq:radmono} proves \eqref{eq:mainest}.
\end{proof}
\begin{remark}
    Combining \eqref{eq:beta-effective} and \eqref{eq:originreflection}, we can get $\gamma$ is explicit.
\end{remark}
\subsection{Proof of Corollary~\ref{cor:firstorder}}
\begin{proof}[Proof of Corollary~\ref{cor:firstorder}]
For any two points $x,y\in \partial B_r$, we have
\[ |k(x)-k(y)|\le  \pi r\,\|\nabla_Tk\|_{L^\infty(B_1)}\le \pi\|\nabla_Tk\|_{L^\infty(B_1)}.\]
Thus, we obtain
\begin{equation}\label{eq:k-1}
    \osc_{\partial B_r}k \le\pi\|\nabla_Tk\|_{L^\infty(B_1)}.
\end{equation}
Moreover, for $e\in\Sn$ and $0\le\rho<r<1$,
\begin{equation}\label{eq:k-2}
    k(re)-k(\rho e)=\int_\rho^r\partial_rk(te)\,dt\le(r-\rho)\|(\partial_rk)_+\|_{L^\infty(B_1)}.
\end{equation}
Hence, combining \eqref{eq:k-1} and \eqref{eq:k-2}, we get $\mathcal D(k)\le\pi\mathcal D_1(k)$.

Finally, by Theorem \ref{thm:main}, we have
\[\mathcal D(u)\le C\mathcal D(k)^{\gamma}\le C\mathcal D_1(k)^{\gamma}.\]
\end{proof}
\subsection{Proof of Theorem~\ref{thm:signchanging}}
\begin{proof}[Proof of Theorem~\ref{thm:signchanging}]
The main proof is the same as that of Theorem~\ref{thm:main}. Hence, we only emphasize the differences. We use the same notation as in the proof of Theorem~\ref{thm:main}.

Set
\[\varepsilon:=\osc_{B_1}k.\]
If $\varepsilon=0$, then $k$ is constant. The qualitative symmetry theorem \cite[Theorem~1.1]{JW16}, applied to the spatially independent nonlinearity $k\,g(u)$, gives $\mathcal D(u)=0$. Hence, we assume $\varepsilon>0$ below. As in the proof of Theorem~\ref{thm:main}, defects above a fixed threshold are absorbed by the uniform bound on $u$, so only the small-defect regime needs to be considered.

First, to show that $\Lambda\neq\emptyset$, we just replace \eqref{eq:kdiff} with
\[
|k(x)-k(x^\lambda)|\le\varepsilon.
\]

Next, we verify the fixed-range step $\lambda_*\le1/4$. Set
\[
d_*:=\frac{1-\lambda_1}{4},\qquad m_*:=\min_{t\in[d_*,1]}F(t)>0.
\]
Choose a fixed $\eta>0$ such that
\begin{equation}\label{eq:eta-2}
    \eta<\min\left\{\frac{1-\lambda_1}{6},\frac1{20},\frac{m_0}{2C_{n}}\right\},\qquad C_H(5\eta/2)^s\le \frac{m_*}{2},
\end{equation}
where $C_n$ denotes the geometric constant in \eqref{eq:Dmeasure}. 

For all $x\in B_{\eta/2}(z_\eta)$, since $1-|x|\le5\eta/2$, it follows from \eqref{eq:Fbound}, \eqref{eq:upperMain}, \eqref{eq:refdistfixed} and \eqref{eq:eta-2} that
\begin{equation}\label{udiff-2}
    u_{\lambda_*}-u\ge F(1-|x^{\lambda_*}|)-C_{H}(1-|x|)^{s} \geq m^{*}-C_{H}\left(\frac52\eta\right)^{s}\geq m_*/2.
\end{equation}
Replacing \eqref{eq:positiveballfixed} with \eqref{udiff-2}, we can obtain \eqref{eq:starquarter}.

Finally, it remains to quantify how close $\lambda_*$ is to the origin. Define the nondecreasing lower envelope
\begin{equation}\label{eq:Flower}
\underline F(r):=\min_{t\in[r,1]}F(t),\qquad0<r\le1.
\end{equation}
Then $\underline F(r)>0$ for $r>0$. Fix $\theta_0>0$, depending only on the data, such that
\begin{equation}\label{eq:theta-1.3}
\theta_0\le\min\left\{\frac{1-\lambda_1}{6}, \frac{2m_{0}}{C_{n}}\right\},
\end{equation}
where $C_{n}$ is used in \eqref{eq:Pmu}. Set
\begin{equation}\label{eq:etaF}
\eta(r):=\min\left\{\theta_0r,\left(\frac{\underline F(r)}{2C_H(5/2)^s}\right)^{1/s}\right\},
\qquad0<r\le\frac14,
\end{equation}
and
\begin{equation}\label{eq:PsiF}
\Psi(r):=\eta(r)^\beta\underline F(r).
\end{equation}
The functions $\eta$ and $\Psi$ are nondecreasing, $\Psi(r)>0$ for $r>0$,and $\Psi(r)\to0$ as $r\downarrow0$.

Assume $\lambda_*>0$ and choose $\eta=\eta(\lambda_*)$. For $x\in B_{\eta/2}(z_{\eta})$, it follows from \eqref{eq:Fbound}, \eqref{eq:z-eta} and \eqref{eq:etaF} that
\[
u(x^{\lambda_*})\ge\underline F(\lambda_*),\qquad u(x)\le\frac12\underline F(\lambda_*),
\]
and hence
\begin{equation}\label{eq:SCball}
\fint_{B_{\eta/2}(z_\eta)}(u_{\lambda_*}-u)\,dx\ge\frac12\underline F(\lambda_*).
\end{equation}
Theorem~\ref{thm:propagation}, applied to $v=u_{\lambda_*}-u+A\varepsilon$, now yields
\begin{equation}\label{eq:SCcore}
\inf_{\Sigma_{\lambda_*,\eta}}(u_{\lambda_*}-u)\ge c\Psi(\lambda_*)-C\varepsilon.
\end{equation}
Suppose, to the contrary, that $c\Psi(\lambda_*)\geq2C\varepsilon$, with
the constants in \eqref{eq:SCcore}. Then
\[
\inf_{\Sigma_{\lambda_*,\eta}}(u_{\lambda_*}-u)\ge c_0\Psi(\lambda_*)>0
\]
for a fixed $c_0>0$. Choose $\sigma_0>0$ so that
\[
C_H(2\sigma_0)^s\le\frac{c_0}{2}\Psi(\lambda_*),\qquad\sigma_0<\min\{\eta,\lambda_*/2\}.
\]
Thus, for all $\lambda_*-\sigma_0\le\mu\le\lambda_*$, \eqref{A-3} holds and contradicts the minimality of $\lambda_*$. We conclude that
\begin{equation}\label{eq:PsiBound}
\Psi(\lambda_*)\le C\varepsilon.
\end{equation}

For sufficiently small $t\ge0$, define
\begin{equation}\label{eq:rF}
r_F(t):=\sup\left\{r\in(0,1/4]:\Psi(r)\le Ct\right\},
\end{equation}
with $r_F(t)=0$ if the set is empty. Since $\Psi(r)>0$ for each fixed $r>0$, $r_F(t)\to0$ as $t\downarrow0$, and \eqref{eq:PsiBound} gives $\lambda_*\le r_F(\varepsilon)$. 

We now spell out the last comparison. Fix a direction and write $r_\varepsilon:=r_F(\varepsilon)$. At the plane $\lambda=r_\varepsilon$, the definition of $\Lambda$ gives
\begin{equation}\label{eq:u-22}
    \|(u-u_{r_\varepsilon})_+\|_{L^\infty(\Sigma_{r_\varepsilon})}\le A\varepsilon.
\end{equation}
If $x_1>r_\varepsilon$, combining \eqref{eq:upperMain} and \eqref{eq:u-22}, we obtain
\begin{equation}
\begin{aligned}
u(x_1,x')-u(-x_1,x')={}&\bigl[u(x_1,x')-u(2r_\varepsilon-x_1,x')\bigr]\\
&+\bigl[u(2r_\varepsilon-x_1,x')-u(-x_1,x')\bigr]\\
\leq{} &A\varepsilon+C_H(2r_\varepsilon)^s.
\end{aligned}
\end{equation}
If $0<x_1\le r_\varepsilon$, \eqref{eq:upperMain} gives
\begin{equation}
    u(x_1,x')-u(-x_1,x')\leq C_H(2r_\varepsilon)^s.
\end{equation}
Similar with \eqref{m-1} and \eqref{m-2}, we can obtain that
\begin{equation}\label{eq:SCfinalraw}
\mathcal D(u)\le C\left(\varepsilon+r_F(\varepsilon)^s\right):=C\omega_F(\epsilon).
\end{equation}
\end{proof}

\section{Consequences, sharpness, and extensions}\label{sec:extensions}

\subsection{Two direct consequences}
Before discussing extensions, we record two consequences that clarify the meaning of the defect estimate.

\begin{corollary}\label{cor:radialprofile}
Under the assumptions of Theorem~\ref{thm:main}, there exists a bounded nonincreasing function $U:[0,1]\to\R$ such that
\begin{equation}\label{eq:profileclose}
\|u-U(|\cdot|)\|_{L^\infty(B_1)}\le C\mathcal D(k)^\gamma.
\end{equation}
Moreover
\begin{equation}\label{eq:centermax}
0\le\max_{\overline{B_1}}u-u(0)\le C\mathcal D(k)^\gamma.
\end{equation}
\end{corollary}
\begin{proof}
Let $\eta:=\mathcal D(u)$ and for $0<r<1,$ we define 
\[
m(r):=\fint_{\partial B_r}u\,dS\quad,
\]
with $m(0):=u(0)$. For all $0\le r<t<1$ and every $e\in\Sn$, we have  $|u(re)-m(r)|\le\eta$, $u(te)-u(re)\le\eta$ and $m(t)-m(r)\le\eta$. Since $u\in C^s(\R^n)$ and $u=0$ on $\partial B_1$, we have $m(r)\to0$ as $r\uparrow1$. Define
\[
U(r):=\sup_{r\le t<1}m(t)\quad(0\le r<1),
\]
with $U(1):=0$. Then $U$ is nonincreasing and $0\le U(r)-m(r)\le\eta$. Hence
\begin{equation}\label{u-sym}
|u(re)-U(r)|\le2\eta.
\end{equation}
Combining \eqref{u-sym} and Theorem~\ref{thm:main},  we get  \eqref{eq:profileclose}. Since $u(x)\le u(0)+\eta$ for every $x\in B_1$, we finally obtain \eqref{eq:centermax}.
\end{proof}

\begin{corollary}\label{cor:families}
Fix $n\ge2$, $0<s<1$, $C_0\ge1$, and a nonnegative $g\in C^{0,1}_{\mathrm{loc}}([0,\infty))$. Let $k_j\in C(B_1;[0,\infty))\cap L^\infty(B_1)$ satisfy
\[
\sup_j\|k_j\|_{L^\infty(B_1)}<\infty,\qquad\mathcal D(k_j)\longrightarrow0.
\]
For each $j$, let $u_j$ be a classical solution of
\[
(-\Delta)^su_j=k_j(x)g(u_j)\quad\text{in }B_1,\qquad u_j=0\quad\text{in }\R^n\setminus B_1,
\]
satisfying
\[
C_0^{-1}\le\|u_j\|_{L^\infty(B_1)}\le C_0.
\]
Then
\begin{equation}\label{eq:familydefect}
\mathcal D(u_j)\le C\mathcal D(k_j)^\gamma\longrightarrow0,
\end{equation}
with $C$ and $\gamma$ independent of $j$. Moreover, $(u_j)$ is precompact in
$C(\overline{B_1})$, and every uniform limit is radial and nonincreasing in
$|x|$.
\end{corollary}
\begin{remark}
    For Corollary \ref{cor:radialprofile} and \ref{cor:families}, similar results hold for nonseparable case.
\end{remark}

\begin{proof}
The uniform estimate \eqref{eq:familydefect} is Theorem~\ref{thm:main}, because all quantities on which its constants depend are bounded uniformly in $j$. Also
\[
\|k_jg(u_j)\|_{L^\infty(B_1)}\le\left(\sup_j\|k_j\|_{L^\infty(B_1)}\right)\|g\|_{L^\infty([0,C_0])}.
\]
Proposition~\ref{prop:boundaryreg} therefore gives a uniform $C^s(\R^n)$ bound. Arzel\`a--Ascoli yields precompactness in $C(\overline{B_1})$. If $u_{j_\ell}\to u_*$ uniformly, then for $|x|=|y|$,
\[
|u_*(x)-u_*(y)| =\lim_{\ell\to\infty}|u_{j_\ell}(x)-u_{j_\ell}(y)|=0,
\]
and, for $e\in\Sn$ and $0\le\rho<r<1$,
\[
u_*(re)-u_*(\rho e) =\lim_{\ell\to\infty} \bigl(u_{j_\ell}(re)-u_{j_\ell}(\rho e)\bigr) \le0.
\]
Thus $u_*$ is radial and nonincreasing.
\end{proof}

\subsection{Nonseparable nonlinearities}
From the proof for Theorem \ref{thm:main} and \ref{thm:signchanging}, the factorization $k(x)g(u)$ is not used at a structural level. Thus, the next result gives a nonseparable formulation and records the uniform spatial defect that replaces $\mathcal D(k)$. 
%Here, the supremum over $t$ is essential because the reflected equation evaluates the spatial difference at $t=u_\lambda(x)$. 
The local quantitative theory also uses uniform spatial deficits in the solution variable (see \cite[Theorem~1.5]{CCPP24}).

\begin{theorem}\label{thm:nonseparable}
Let $n\ge2$, $0<s<1$, and $C_0\ge1$. Suppose $f:B_1\times[0,C_0]\to[0,\infty)$ is bounded, continuous, and uniformly Lipschitz in its second variable, namely, for all $x\in B_1,\ 0\le t,\tau\le C_0$,
\begin{equation}\label{eq:fuLip}
|f(x,t)-f(x,\tau)|\le L_f|t-\tau|.
\end{equation}
Let $u$ be a classical solution of
\begin{equation}\label{eq:nonsep}
(-\Delta)^su=f(x,u)\quad\text{in }B_1, \qquad u=0\quad\text{in }\R^n\setminus B_1,
\end{equation}
such that
\[
\frac1{C_0}\le\|u\|_{L^\infty(B_1)}\le C_0.
\]
Then
\begin{equation}\label{eq:nonsepest}
\mathcal D(u) \le C\,\mathcal D_x(f;C_0)^\gamma,
\end{equation}
where $C>0$ and $\gamma\in(0,1)$ depend on $n,s,C_0,L_f$ and $\|f\|_{L^\infty(B_1\times[0,C_0])}$.
\end{theorem}

\begin{proof}
Set $\varepsilon:=\mathcal D_x(f;C_0)$. If $\varepsilon=0$, then $x\mapsto f(x,t)$ is radial and nonincreasing for every $t\in[0,C_0]$. Thus, $\mathcal D(u)=0$. If $\varepsilon$ is larger than a fixed positive threshold, \eqref{eq:nonsepest} follows from $\mathcal D(u)\le2C_0$ after enlarging $C$. We may therefore work in the same small-defect regime as in the proof of Theorem~\ref{thm:main}.

In the remainder of the proof, we only show the steps that differ from the proof of Theorem \ref{thm:main}. 

In a cap $\Sigma_\lambda$, set $w_\lambda=u-u_\lambda$. If $u(x)\ne u_\lambda(x)$, define
\begin{equation}\label{eq:clambda-4.2}
c_\lambda(x):=
\begin{cases}
-\frac{f(x,u(x))-f(x,u_\lambda(x))}{u(x)-u_\lambda(x)},&u(x)\ne u_\lambda(x),\\[8pt]
0,&u(x)=u_\lambda(x).
\end{cases}
\end{equation}
The assumption on $f$ gives $|c_\lambda|\le L_f$. Similarly, from \eqref{eq:wEq}, we obtain
\begin{equation}\label{eq:nonsep-reflected}
(-\Delta)^sw_\lambda+c_\lambda w_\lambda=f(x,u_\lambda)-f(x^\lambda,u_\lambda).
\end{equation}
For $x\in\Sigma_\lambda$ one has $|x^\lambda|\le|x|$. Applying the decomposition in \eqref{eq:kdiff} to the function $f(\cdot,t)$ with $t=u_\lambda(x)\in[0,C_0]$ gives
\begin{equation}\label{eq:nonsep-force}
\bigl(f(x,u_\lambda)-f(x^\lambda,u_\lambda)\bigr)_+\le\mathcal D_x(f;C_0).
\end{equation}
Replacing \eqref{eq:cBound} by $|c_\lambda|\le L_f$ and \eqref{eq:forcing} by \eqref{eq:nonsep-force}, we obtain
\[
\lambda_*^{s+\beta}\le C\mathcal D_x(f;C_0).
\]
from the proof of Theorem \ref{thm:main}.
\end{proof}

Signed nonseparable sources can be treated in the same framework once the lower boundary is imposed explicitly, as in Theorem~\ref{thm:signchanging}.

\begin{corollary}\label{cor:signed-nonseparable}
Let $f:B_1\times[0,C_0]\to\R$ be bounded, continuous, and uniformly Lipschitz in its second variable with constant $L_f$. Let $F\in C([0,1];[0,\infty))$ satisfy $F(0)=0$ and $F(t)>0$ for $t>0$, and let $u$ be a classical  solution of
\[
(-\Delta)^su=f(x,u)\quad\text{in }B_1,\qquad u=0\quad\text{in }\R^n\setminus B_1,
\]
such that for all $x\in B_{1}$, $$F(1-|x|)\le u(x)\le C_0.$$ 
Then
\[
\mathcal D(u)\le\omega_F\bigl(\operatorname{osc}_x(f;C_0)\bigr),
\]
where $\omega_F(t)\to0$ as $t\downarrow0$ and the modulus depends only on
$n,s,C_0,F,L_f$ and $\|f\|_\infty$.
\end{corollary}

\subsection{The power scale and the linear-rate question}
Nevertheless, there is a simple test for the admissible power scale: no uniform exponent larger than one is possible, already for the linear Dirichlet problem obtained by taking $g\equiv1$.

\begin{proposition}\label{prop:powerceiling}
Let $n\ge2$ and $0<s<1$. For all $0<\varepsilon\le1/2$, there exist a constant $C_0\ge1$ and a family of smooth positive coefficients $k_\varepsilon$, with $\frac12\le k_\varepsilon\le\frac32$ and corresponding positive classical solutions $u_\varepsilon$ of
\[
(-\Delta)^su_\varepsilon=k_\varepsilon(x)\quad\text{in }B_1,\qquad u_\varepsilon=0\quad\text{in }\R^n\setminus B_1,
\]
satisfying the normalization \eqref{eq:norm} with this fixed $C_0$, such that
\begin{equation}\label{eq:linear-lower}
\mathcal D(k_\varepsilon)=c_1\varepsilon,\qquad\mathcal D(u_\varepsilon)\ge c_2\varepsilon
\end{equation}
for constants $c_1,c_2>0$ independent of $\varepsilon$. Consequently, there is no estimate
\[
\mathcal D(u)\le C\mathcal D(k)^\alpha
\]
can hold uniformly under the hypotheses of Theorem~\ref{thm:main} with any $\alpha>1$.
\end{proposition}

\begin{proof}
Choose a nonradial $h\in C_c^\infty(B_1)$ with $\|h\|_{L^\infty(B_1)}\le1$. Define $G(x,y)=G_1(x,y)$ be the Green function in \eqref{eq:GreenR}. Then, set
\begin{equation}
G_s[f](x):=\int_{B_1}G(x,y)f(y)\,dy.
\end{equation}
Denote
\[
u_0:=\mathcal G_s[1],\qquad v:=\mathcal G_s[h].
\]
Thus, $u_0>0$ is radial, while $v$ cannot be radial. Indeed, if $v$ were radial, then $h=(-\Delta)^sv$ was radial in $B_1$, contrary to the choice of $h$. Hence
\[
c_2:=\sup_{0<r<1}\osc_{\partial B_r}v>0.
\]

For $0<\varepsilon\le1/2$, define
\[
k_\varepsilon:=1+\varepsilon h,\qquad u_\varepsilon:=u_0+\varepsilon v=\mathcal G_s[k_\varepsilon].
\]
Then $1/2\le k_\varepsilon\le3/2$. 
Since $G(x,y)>0$, the comparison principle gives
\[
\frac12u_0\le u_\varepsilon\le\frac32u_0\qquad\text{in }B_1.
\]
Thus, with
\[
C_0:=\max\left\{1,\frac{2}{\|u_0\|_\infty},\frac32\|u_0\|_\infty\right\},
\]
the two-sided $L^\infty$ normalization holds uniformly in $\varepsilon$. Since $k_\varepsilon$ is smooth in $B_1$, interior Schauder regularity gives $u_\varepsilon\in C^\infty_{\rm loc}(B_1)$(see \cite{Stinga19}). Thus, $u_\varepsilon$ is a classical solution of \eqref{eq:main}.

Since
\[
\mathcal D(k_\varepsilon)=\varepsilon\mathcal D(h)=:c_1\varepsilon,
\]
where $c_1>0$ because $h$ is nonradial. $u_0$ is radius symmetric, so
\[
\osc_{\partial B_r}u_\varepsilon=\varepsilon\osc_{\partial B_r}v.
\]
Taking the supremum over $r$ gives $\mathcal D(u_\varepsilon)\ge c_2\varepsilon$, which proves \eqref{eq:linear-lower}. If a uniform estimate with $\alpha>1$ held, then $c_2\varepsilon\le C(c_1\varepsilon)^\alpha$ for all small $\varepsilon$, an impossibility as $\varepsilon\downarrow0$.
\end{proof}

\begin{remark}
Proposition~\ref{prop:powerceiling} shows that a power exponent cannot exceed $1$, but it is an interesting question to identify the universal optimal power  
\begin{equation}\label{eq:linear-question}
\mathcal D(u)\le C\mathcal D(k).
\end{equation}
\end{remark}
{\bf{Acknowledgments.}} 
All authors are partially supported by the National Natural Science Foundation of China (Grant No. W2531006, 12250710674 and 12031012) and the Institute of Modern Analysis-A Frontier Research Center of Shanghai.

\medskip
{\bf AI disclosure statement.}
The authors acknowledge the use of OpenAI to assist with language editing and manuscript preparation. All mathematical results and proofs were reviewed and verified by the authors, who take full
responsibility for the content of the manuscript.

\medskip
{\bf Data availability statement.} No data was used for the research described in the article.

\medskip
{\bf Conflict of interest statement.} On behalf of all authors, the corresponding author states that there is no conflict of interest.


\begin{thebibliography}{99}

\bibitem{ABR99}
A.~Aftalion, J.~Busca and W.~Reichel, \emph{Approximate radial symmetry for overdetermined boundary value problems}, Adv. Differential Equations \textbf{4} (1999), no.~6, 907--932.

\bibitem{Alex62}
A.~D.~Alexandrov,\emph{A characteristic property of spheres}, Ann. Mat. Pura Appl. (4) \textbf{58} (1962), 303--315.

%\bibitem{AC00}
%L. Ambrosio and X. Cabr\'e, \emph{Entire solutions of semilinear elliptic equations in $\bold R^3$ and a conjecture of De Giorgi}, J. Amer. Math. Soc. {\bf 13} (2000), no.~4, 725--739.

\bibitem{BMS18}
B.~Barrios, L.~Montoro and B.~Sciunzi, \emph{On the moving plane method for nonlocal problems in bounded domains}, J. Anal. Math. \textbf{135} (2018), no.~1, 37--57.

%\bibitem{BCN97}
%H. Berestycki, L.~\'A. Caffarelli and L. Nirenberg, \emph{Monotonicity for elliptic equations in unbounded Lipschitz domains}, Comm. Pure Appl. Math. {\bf 50} (1997), no.~11, 1089--1111.

\bibitem{BN91}
H.~Berestycki and L.~Nirenberg, \emph{On the method of moving planes and the sliding method}, Bol. Soc. Brasil. Mat. (N.S.) \textbf{22} (1991), no.~1, 1--37.

\bibitem{Be96}
J. Bertoin, L\'{e}vy Processes, \emph{Cambridge Tracts in Mathematics}, 121 Cambridge University Press, Cambridge, 1996.

%\bibitem{BLW05}
%M.~Birkner, J.~A.~L\'opez-Mimbela and A.~Wakolbinger, \emph{Comparison results and steady states for the Fujita equation with fractional Laplacian}, Ann. Inst. H. Poincar\'e Probab. Statist. \textbf{41} (2005), no.~1, 83--97.

\bibitem{Biswas26}
S.~Biswas, \emph{Symmetry and approximate symmetry for solutions of mixed local--nonlocal singular equations}, J. Differential Equations \textbf{480} (2026), Paper No. 114658.

\bibitem{BG90}
J. P. Bouchard, A. Georges, \emph{Anomalous diffusion in disordered media, Statistical mechanics, models and physical applications}, Physics reports, 195 (1990).

\bibitem{Bucur16}
C.~Bucur, \emph{Some observations on the Green function for the ball in the fractional Laplace framework}, Commun. Pure Appl. Anal. \textbf{15} (2016), no.~2, 657--699.

\bibitem{CGS89}
L.~\'A. Caffarelli, B. Gidas and J. Spruck, \emph{Asymptotic symmetry and local behavior of semilinear elliptic equations with critical Sobolev growth}, Comm. Pure Appl. Math. {\bf 42} (1989), no.~3, 271--297.

%\bibitem{CS07}
%L.~Caffarelli and L.~Silvestre, \emph{An extension problem related to the fractional Laplacian}, Comm. Partial Differential Equations \textbf{32} (2007), no.~7--9, 1245--1260.

\bibitem{CV10}
L.~\'A. Caffarelli and A.~F. Vasseur, \emph{Drift diffusion equations with fractional diffusion and the quasi-geostrophic equation}, Ann. of Math. (2) {\bf 171} (2010), no.~3, 1903--1930.

\bibitem{CCG25}
G.~Ciraolo, M.~Cozzi and M.~Gatti, \emph{A quantitative study of radial symmetry for solutions to semilinear equations in $\R^n$}, J. Math. Pures Appl. \textbf{204} (2025), Paper No. 103755, 45 pp.

\bibitem{CCPP24}
G.~Ciraolo, M.~Cozzi, M.~Perugini and L.~Pollastro, \emph{A quantitative version of the Gidas--Ni--Nirenberg theorem}, J. Funct. Anal. \textbf{287} (2024), no.~9, Paper No. 110585, 29 pp.

\bibitem{CDPPV23}
G.~Ciraolo, S.~Dipierro, G.~Poggesi, L.~Pollastro and E.~Valdinoci, \emph{Symmetry and quantitative stability for the parallel surface fractional torsion problem}, Trans. Amer. Math. Soc. \textbf{376} (2023), no.~5, 3515--3540.

\bibitem{CFMN18}
G.~Ciraolo, A.~Figalli, F.~Maggi and M.~Novaga, \emph{Rigidity and sharp stability estimates for hypersurfaces with constant and almost-constant nonlocal mean curvature}, J. Reine Angew. Math. \textbf{741} (2018), 275--294.

\bibitem{CL25}
G.~Ciraolo and X.~Li,\emph{A quantitative symmetry result for $p$-Laplace equations with discontinuous nonlinearities}, Math. Ann. \textbf{392} (2025), no.~2, 2131--2155.

\bibitem{CMV16}
G.~Ciraolo, R.~Magnanini and V.~Vespri, \emph{H\"older stability for Serrin's overdetermined problem}, Ann. Mat. Pura Appl. (4) \textbf{195} (2016), no.~4, 1333--1345.

\bibitem{CR18}
G. Ciraolo and A. Roncoroni, \emph{The method of moving planes: a quantitative approach}, in {\it Bruno Pini Mathematical Analysis Seminar 2018}, 41--77, Bruno Pini Math. Anal. Semin., \textbf{9}, Univ. Bologna, Alma Mater Stud., Bologna
%G.~Ciraolo and A.~Roncoroni,\emph{The method of moving planes: a quantitative approach}, Bruno Pini Math. Anal. Semin. \textbf{9} (2018), 41--77.

\bibitem{CV18}
G.~Ciraolo and L.~Vezzoni, \emph{A sharp quantitative version of Alexandrov's theorem via the method of moving planes}, J. Eur. Math. Soc. (JEMS) \textbf{20} (2018), no.~2, 261--299.

\bibitem{CGY26}
H.~B. Chen, C. Gui and R. Yao, \emph{Uniqueness of critical points of the second Neumann eigenfunctions on triangles}, Invent. Math. {\bf 244} (2026), no.~1, 299--353.

%\bibitem{CGL26}
%W.~Chen, Y.~Guo and C.~Li, \emph{Methods in studying qualitative properties of fractional equations}, preprint, arXiv:2601.19783, 2026.

\bibitem{CL91}
W. Chen and C. Li, \emph{Classification of solutions of some nonlinear elliptic equations}, Duke Math. J. {\bf 63} (1991), no.~3, 615--622

%\bibitem{CL10} 
%W. Chen and C. Li, \emph{ Methods on nonlinear elliptic equations}, AIMS Series on Differential Equations \& Dynamical Systems, 4, Am. Inst. Math. Sci. (AIMS), Springfield, MO, 2010.

%\bibitem{CL18}
%W. Chen and C. Li, \emph{Maximum principles for the fractional $p$-Laplacian and symmetry of solutions}, Adv. Math. {\bf 335} (2018), 735--758.

%\bibitem{CLG17}
%W. Chen, C. Li and G. Li, \emph{Maximum principles for a fully nonlinear fractional order equation and symmetry of solutions}, Calc. Var. Partial Differential Equations {\bf 56} (2017), no.~2, Paper No. 29, 18 pp.

\bibitem{CLL17}
W.~Chen, C.~Li and Y.~Li, \emph{A direct method of moving planes for the fractional Laplacian}, Adv. Math. \textbf{308} (2017), 404--437.

\bibitem{CL005}
W. Chen, C. Li and B. Ou, \emph{Qualitative properties of solutions for an integral equation}, Discrete Contin. Dyn. Syst. {\bf 12} (2005), no.~2, 347--354.

\bibitem{CLO06}
W.~Chen, C.~Li and B.~Ou, \emph{Classification of solutions for an integral equation}, Comm. Pure Appl. Math. \textbf{59} (2006), no.~3, 330--343.

\bibitem{CLZ17}
W. Chen, Y. Li and R. Zhang, \emph{A direct method of moving spheres on fractional order equations}, J. Funct. Anal. {\bf 272} (2017), no.~10, 4131--4157.

\bibitem{C06}
P. Constantin, \emph{Euler equations, Navier-Stokes equations and turbulence, in Mathematical Foundation of Turbulent Viscous Flows}, Vol. 1871 of Lecture Notes in Math. Springer, Berlin, 2006.

%\bibitem{Cheng17}
%T.~Cheng, \emph{Monotonicity and symmetry of solutions to fractional Laplacian equations}, Discrete Contin. Dyn. Syst. \textbf{37} (2017), no.~7, 3587--3599.

\bibitem{DQ18}
W. Dai and G. Qin, \emph{Classification of nonnegative classical solutions to third-order equations}, Adv. Math. {\bf 328} (2018), 822--857.
%\bibitem{DQ23}
%W. Dai and G. Qin, Liouville-type theorems for fractional and higher-order H\'enon-Hardy type equations via the method of scaling spheres, Int. Math. Res. Not. IMRN {\bf 2023}, no.~11, 9001--9070.

\bibitem{Dancer92}
E.~N.~Dancer, \emph{Some notes on the method of moving planes}, Bull. Austral. Math. Soc. \textbf{46} (1992), no.~3, 425--434.

\bibitem{DKW11}
M.~A. del~Pino, M. Kowalczyk and J. Wei, \emph{On De Giorgi's conjecture in dimension $N\geq 9$}, Ann. of Math. (2) {\bf 174} (2011), no.~3, 1485--1569.

\bibitem{DCKP14}
A.~Di Castro, T.~Kuusi and G.~Palatucci,\emph{Nonlocal Harnack inequalities}, J. Funct. Anal. \textbf{267} (2014), no.~6, 1807--1836.

\bibitem{DGSPV25}
S.~Dipierro, J.~Gon\c{c}alves da Silva, G.~Poggesi and E.~Valdinoci, \emph{A quantitative Gidas--Ni--Nirenberg-type result for the $p$-Laplacian via integral identities}, J. Funct. Anal. \textbf{289} (2025), no.~10, Paper No. 111108, 38 pp.

\bibitem{DKTV25}
S.~Dipierro, M.~Kwa\'snicki, J.~Thompson and E.~Valdinoci, \emph{The nonlocal Harnack inequality for antisymmetric solutions: an approach via Bochner's relation and harmonic analysis}, Comm. Partial Differential Equations \textbf{50} (2025), no.~8, 1074--1098.

\bibitem{DPTV23Serrin}
S.~Dipierro, G.~Poggesi, J.~Thompson and E.~Valdinoci, \emph{Quantitative stability for overdetermined nonlocal problems with parallel surfaces and investigation of the stability exponents}, J. Math. Pures Appl. (9) {\textbf 188} (2024), 273--319.

\bibitem{DPTV24}
S.~Dipierro, G.~Poggesi, J.~Thompson and E.~Valdinoci, \emph{The role of antisymmetric functions in nonlocal equations}, Trans. Amer. Math. Soc. \textbf{377} (2024), no.~3, 1671--1692.

%\bibitem{DPTV24exp}
%S.~Dipierro, G.~Poggesi, J.~Thompson and E.~Valdinoci, \emph{Quantitative stability for overdetermined nonlocal problems with parallel surfaces and investigation of the stability exponents}, J. Math. Pures Appl. \textbf{188} (2024), 273--319.

\bibitem{DTV23}
S.~Dipierro, J.~Thompson and E.~Valdinoci, \emph{On the Harnack inequality for antisymmetric $s$-harmonic functions}, J. Funct. Anal. \textbf{285} (2023), no.~1, Paper No. 109917, 49 pp.

%\bibitem{DPV12}
%E.~Di Nezza, G.~Palatucci and E.~Valdinoci, \emph{Hitchhiker's guide to the fractional Sobolev spaces}, Bull. Sci. Math. \textbf{136} (2012), no.~5, 521--573.

\bibitem{Dou16}
M.~Dou, \emph{A direct method of moving planes for fractional Laplacian equations in the unit ball}, Commun. Pure Appl. Anal. \textbf{15} (2016), no.~5, 1797--1807.

%\bibitem{DKK}
%B. Dyda, A. Kuznetsov and M. Kwa\'snicki,\emph{Fractional Laplace operator and Meijer G-function},Constr. Approx. \textbf{45} (2017), no.~3, 427--448.

\bibitem{FQJ12}
P.~L. Felmer, A. Quaas and J. Tan, \emph{Positive solutions of the nonlinear Schr\"odinger equation with the fractional Laplacian}, Proc. Roy. Soc. Edinburgh Sect. A {\bf 142} (2012), no.~6, 1237--1262.

%\bibitem{FS2020}
%A. Figalli and J. Serra, \emph{On stable solutions for boundary reactions: a De Giorgi-type result in dimension $4+1$}, Invent. Math. {\bf 219} (2020), no.~1, 153--177

%\bibitem{FJ15}
%M.~M.~Fall and S.~Jarohs,\emph{Overdetermined problems with fractional Laplacian}, ESAIM Control Optim. Calc. Var. \textbf{21} (2015), no.~4, 924--938.

\bibitem{FW14}
P.~Felmer and Y.~Wang, \emph{Radial symmetry of positive solutions to equations involving the fractional Laplacian}, Commun. Contemp. Math. \textbf{16} (2014), no.~1, 1350023, 24 pp.

\bibitem{Gatti25}
M.~Gatti, \emph{Approximate radial symmetry for $p$-Laplace equations via the moving planes method}, Calc. Var. Partial Differential Equations \textbf{64} (2025), Paper No. 261, 56 pp.

\bibitem{GSW26}
M.~Gatti, J.~Scheuer and T.~Weth, \emph{Fractional Dirichlet problems with an overdetermined non-local Neumann condition}, Proc. Lond. Math. Soc. (3) \textbf{132} (2026), no.~4,  Paper No. e70149, 34 pp.

%\bibitem{GG98}
%N.~A. Ghoussoub and C. Gui, \emph{On a conjecture of De Giorgi and some related problems}, Math. Ann. {\bf 311} (1998), no.~3, 481--491.

%\bibitem{GG03}
%N.~A. Ghoussoub and C. Gui, On De Giorgi's conjecture in dimensions 4 and 5, Ann. of Math. (2) {\bf 157} (2003), no.~1, 313--334.

\bibitem{GNN79}
B.~Gidas, W.-M.~Ni and L.~Nirenberg,\emph{Symmetry and related properties via the maximum principle}, Comm. Math. Phys. \textbf{68} (1979), no.~3, 209--243.

\bibitem{GNN81}
B.~Gidas, W.-M.~Ni and L.~Nirenberg, \emph{Symmetry of positive solutions of nonlinear elliptic equations in $\R^n$}, in: Mathematical Analysis and Applications, Part A, Adv. Math. Suppl. Stud. \textbf{7a}, Academic Press, New York-London, 1981, 369--402.

%\bibitem{GSGL23}
%C. Giulio, D. Serena, P. Giorgio, P. Luigim and V. Enrico, \emph{Symmetry and quantitative stability for the parallel surface fractional torsion problem}, Trans. Amer. Math. Soc. {\bf 376} (2023), no.~5, 3515--3540

%\bibitem{GM18}
%C. Gui and A. Moradifam, \emph{The sphere covering inequality and its applications}, Invent. Math. {\bf 214} (2018), no.~3, 1169--1204.

\bibitem{Jarohs16}
S.~Jarohs, \emph{Symmetry of solutions to nonlocal nonlinear boundary value problems in radial sets}, NoDEA Nonlinear Differential Equations Appl. \textbf{23} (2016), no.~3, Art. 32, 22 pp.

\bibitem{JW16}
S.~Jarohs and T.~Weth, \emph{Symmetry via antisymmetric maximum principles in nonlocal problems of variable order}, Ann. Mat. Pura Appl. (4) \textbf{195} (2016), no.~1, 273--291.

\bibitem{Li96}
C. Li, \emph{Local asymptotic symmetry of singular solutions to nonlinear elliptic equations}, Invent. Math. {\bf 123} (1996), no.~2, 221--231

\bibitem{LZ11}
G. Lu and J. Zhu, \emph{Symmetry and regularity of extremals of an integral equation related to the Hardy-Sobolev inequality}, Calc. Var. Partial Differential Equations {\bf 42} (2011), no.~3-4, 563--577.

%\bibitem{PS25}
%L.~Pollastro and N.~Soave, \emph{Antisymmetric maximum principles and Hopf's lemmas for the logarithmic Laplacian, with applications to symmetry results}, Ann. Mat. Pura Appl. (4) \textbf{204} (2025), 1827--1845.

\bibitem{ROS14}
X.~Ros-Oton and J.~Serra, \emph{The Dirichlet problem for the fractional Laplacian: regularity up to the boundary}, J. Math. Pures Appl. (9) \textbf{101} (2014), no.~3, 275--302.

%\bibitem{RJ14}
%X. Ros-Oton and J. Serra, \emph{The extremal solution for the fractional Laplacian}, Calc. Var. Partial Differential Equations {\bf 50} (2014), no.~3-4, 723--750.

%\bibitem{ROS16stable}
%X.~Ros-Oton and J.~Serra, \emph{Regularity theory for general stable operators}, J. Differential Equations \textbf{260} (2016), no.~12, 8675--8715.

%\bibitem{ROS17boundary}
%X.~Ros-Oton and J.~Serra, \emph{Boundary regularity estimates for nonlocal elliptic equations in $C^1$ and $C^{1,\alpha}$ domains}, Ann. Mat. Pura Appl. (4) \textbf{196} (2017), no.~5, 1637--1668.

\bibitem{Rosset94}
E.~Rosset,\emph{An approximate Gidas--Ni--Nirenberg theorem}, Math. Methods Appl. Sci. \textbf{17} (1994), no.~13, 1045--1052.

\bibitem{Savin09}
O.~V. Savin, Regularity of flat level sets in phase transitions, Ann. of Math. (2) {\bf 169} (2009), no.~1, 41--78.

\bibitem{Serrin71}
J.~Serrin, \emph{A symmetry problem in potential theory}, Arch. Rational Mech. Anal. \textbf{43} (1971), 304--318.

\bibitem{Stinga19}
P.~R.~Stinga,\emph{User's guide to the fractional Laplacian and the method of semigroups}, in: Fractional Differential Equations, Handbook of Fractional Calculus with Applications, Vol.~2, De Gruyter, Berlin--Boston, 2019, 235--266.

\bibitem{TZ06} 
V.~E. Tarasov and G.~M. Zaslavsky, \emph{Fractional dynamics of systems with long-range interaction}, Commun. Nonlinear Sci. Numer. Simul. {\bf 11} (2006), no.~8, 885--898.

\bibitem{WQY24}
Y. Wang, Y. Qiu and Q. Yin, \emph{The radial symmetry of positive solutions for semilinear problems involving weighted fractional Laplacians}, Acta Math. Sci. Ser. B (Engl. Ed.) {\bf 44} (2024), no.~3, 1020--1035

\bibitem{Zhang15}
L. Zhang, \emph{Symmetry of solutions to semilinear equations involving the fractional Laplacian}, Commun. Pure Appl. Anal. {\bf 14} (2015), no.~6, 2393--2409

\end{thebibliography}
\end{document}